\documentclass[12pt]{amsart}
\usepackage{url}
\usepackage[margin=1in]{geometry} 
\usepackage{color}    
\usepackage{latexsym, amssymb, amscd, amsfonts, amsthm, amsmath, graphicx}
\usepackage{url}
\usepackage[margin=1in]{geometry}  % set the margins to 1in on all sides
\usepackage{epsfig}
\usepackage{color}                 % better theorem environments
\usepackage{epsfig}
\usepackage{mathrsfs}

\theoremstyle{plain}
\newtheorem{thm}{Theorem}[section]

\newtheorem{theorem}{Theorem}[section]
\newtheorem*{theorem*}{Theorem}
\newtheorem*{lemma*}{Lemma}
\newtheorem{proposition}[thm]{Proposition}
\newtheorem{lemma}[thm]{Lemma}
\newtheorem{corollary}[thm]{Corollary}
\theoremstyle{definition}

\newtheorem{example}{Example}
\theoremstyle{remark}

\newtheorem{remark}{Remark}

\DeclareMathOperator{\id}{\rm{Id}}
\newcommand{\esssup}{\operatorname{ess\;sup}}

\newcommand{\vtheta}{{\boldsymbol{\theta}}}
\newcommand{\ve}{{\mathbf{e}}}
\newcommand{\loc}{{\mathrm{loc}}}
\newcommand{\tr}{{\mathrm{Tr}}}
\newcommand{\R}{\mathbb{R}}      % for Real numbers
\newcommand{\Z}{\mathbb{Z}}      % for integer
\newcommand{\C}{\mathbb{C}}
\newcommand{\T}{\mathbb{T}}
\newcommand{\N}{\mathbb{N}}
\newcommand{\bR}{\mathbb{R}}      % for Real numbers
\newcommand{\bT}{\mathbb{T}}

\newcommand{\sB}{\mathscr{B}}

\newcommand{\cM}{\mathscr{M}}
\newcommand{\cC}{\mathscr{C}}
\newcommand{\E}{\mathbb{E}}     %mathematical expectation
\newcommand{\Tr}{\rm{Tr }}     %trace
\newcommand{\TV}{{\rm{TV}}}

\newcommand{\Id}{{\text{Id}}}
\newcommand{\va}{{\mathbf{a}}}
\newcommand{\vb}{{\mathbf{b}}}
\newcommand{\ta}{{\tilde{a}}}
\newcommand{\tb}{{\tilde{b}}}
\newcommand{\tO}{{\tilde{0}}}
\newcommand{\vs}{{\mathbf{s}}}

\newcommand{\vO}{{\mathbf{0}}}

\newcommand{\vz}{{\mathbf{z}}}
\newcommand{\cR}{{\mathscr{R}}}

\newcommand{\zed}{\mathbb{Z}}
\newcommand{\rs}{\mathrm{s}}
\newcommand{\one}{\mathbf{1}}
\newcommand{\balpha}{{\boldsymbol{\alpha}}}
\newcommand{\bomega}{{\boldsymbol{\omega}}}
\newcommand{\bepsilon}{{\boldsymbol{\epsilon}}}

\title[The cut-off phenomenon in Rosenthal's walk]{Cut-off phenomenon in the uniform plane Kac walk} 
\author{Bob Hough and Yunjiang Jiang}
\subjclass[2010]{Primary 60J05, 60B15, 20C15, 43A75}
\keywords{Random walk on a group, cut-off phenomenon, character theory, saddle point analysis}

\begin{document}

\begin{abstract}We consider an analogue of the Kac random walk on the special orthogonal group $SO(N)$, in which at each step a random rotation is performed in a randomly chosen 2-plane of $\bR^N$. We obtain sharp
  asymptotics for the rate of convergence in total variance distance,
  establishing a cut-off phenomenon in the large $N$ limit. In the special case where the angle of rotation is deterministic this confirms a conjecture of Rosenthal \cite{Rosenthal}. Under mild conditions we also establish a cut-off for  convergence of the walk to stationarity under the $L^2$ norm.  Depending on the distribution of the randomly chosen angle of rotation, several surprising features emerge.  For instance, it is sometimes the case that the mixing times differ in the total variation and $L^2$ norms. Our estimates use an integral representation of the characters of the special orthogonal group together with saddle point analysis.  
\end{abstract}

\maketitle
  
\section{Introduction}
Asymptotic analysis of the mixing time of Markov chains is an emerging field, and while discrete space Markov chains have a healthy growing literature, there are still comparatively few continuous state chains where the mixing time has been  determined.  Among the most natural such chains is the random walk on the special orthogonal group introduced by Mark Kac \cite{kac}, modeling the velocities of a large number $N$ of particles making elastic collisions.  Briefly, at each step of the walk a uniform angle $\theta \in \bR/2\pi\zed = \T$ is chosen, and two velocities $v_i$ and $v_j$ are chosen at random and updated according to the rule
\begin{align*}
 v_i' &= v_i\cos \theta  + v_j\sin \theta \\
 v_j' &= -v_i\sin \theta  + v_j\cos \theta .
\end{align*}
That is, the vector of velocities $(v_1, ..., v_N)$ is multiplied by a matrix from the special orthogonal group $SO(N)$ which consists of a rotation by a randomly chosen angle, in the randomly chosen coordinate 2-plane, $e_i \wedge e_j$.  

The asymptotic mixing time of Kac's walk has received quite a bit of attention, going back at least to the paper of Diaconis and Saloff-Coste \cite{diaconis_saloff-coste}, but remains only incompletely understood.  Besides the evident difficulty that the steps do not commute, the most significant challenge is that the bulk of the spectrum of the transition kernel is not known (although the spectral gap has been determined, see  \cite{maslen}, \cite{janvresse}, \cite{carlen_carvalho_loss}).  Even granting the spectrum,  the walk's density is not in $L^2$ after any finite number of steps, which complicates the use of standard spectral techniques.  In fact, only recently the second author \cite{TVKac} has given the first polynomial bound for the mixing time in the total variation metric, but the bound $O\left(N^5(\log N)^2\right)$ it is still far from the expected  bound $O\left(N^2\right)$.

The purpose of this article is to study a related but simplified model of the Kac walk  for which we are able to prove precise results.  In the \emph{uniform plane Kac walk} or \emph{Rosenthal's walk} at each step we perform a rotation by a random angle $\theta$, but now with the plane of rotation chosen uniformly from the space of $2$-planes in $\R^N$, $\Lambda^2(S^{N-1})$. Formally, let $\xi$ be a Borel probability measure on $\T$ and define the block diagonal matrix
\begin{align*}
 R(1,2;\theta) := \left( \begin{array}{ccc} \cos \theta & -\sin \theta
     &  \\ \sin 
\theta & \cos \theta &\\ & & I_{N-2}  \end{array} \right).
\end{align*}
The transition kernel of our walk  is the map $P_\xi: L^2(SO(N)) \to L^2(SO(N))$,
\[
 P_\xi f(X) = \int_{\T} \int_{SO(N)} f\left(U R(1,2;\theta) U^{-1} X\right) d\nu(U) d\xi(\theta)
\]
where $d\nu$ denotes the probability Haar measure on $SO(N)$. To complete the analogy, if we replace the integral over $U \in SO(N)$ by an integral over the permutation matrices, we get exactly the Kac walk.

Throughout this work we  restrict attention to odd $(N = 2n+1)$ orthogonal groups.  We expect that the results carry over to even orthogonal groups with straightforward modifications, but have not checked the details carefully.  Note that $R(1,2;\theta)$ and $R(1,2;-\theta)$ are conjugate in $SO(N)$ for $N \geq 3$, so henceforth we  restrict attention to measures $\xi$ supported in $\bT_0 = [0, \pi]$.

Given two Borel probability measures $\mu$ and $\nu$ defined on the common space $(\Omega = SO(2n+1), \sB)$ with $\sB$ the Borel $\sigma$-algebra, recall  that the total variation distance between them is induced from a norm $\|.\|_\TV := \|.\|_{\TV(\Omega)}$ on signed measures defined by
\begin{align*}
\| \mu - \nu \|_\TV  = \sup_{B \in \sB} \mu(B) - \nu(B).
\end{align*}
This is also known as the $L^1$ distance since when $\left|\frac{d\mu}{d\nu}\right| < \infty$, 
\begin{align*}
\| \mu - \nu \|_\TV = \int_\Omega \left|\frac{d\mu}{d\nu}-1\right| d \nu.
\end{align*}
From a probabilistic point of view, the total variation distance is a natural one, since it is well-defined on all measures. Throughout, we take $\nu$ to be the Haar probability measure on $SO(2n+1)$. In the case that $\mu$ has an $L^2$ density with respect to $\nu$, we
 define their $L^2$ distance 
\[ \|\mu - \nu\|_{L^2} = \left( \int_\Omega \left(\frac{d|\mu -
      \nu|}{d\nu}\right)^2 d\nu\right)^{\frac{1}{2}}. 
\] 
Otherwise we set the $L^2$ distance to be $\infty$.
\begin{theorem} \label{L^1 cut-off}
 Let $\xi$ be any Borel probability  measure on $\T_0$
excluding the delta measure at 0.
Define $\sigma(\theta) = \sin\frac{\theta}{2}$ and
\[
 \xi\left(\sigma^2\right) = \int_{\T_0} \sigma^2(\theta) d\xi(\theta).
\]
The uniform plane Kac walk $P_\xi$ with initial state at the identity on $SO(2n+1)$ has
total variation cut-off at $t = \frac{n \log n}{2 \xi(\sigma^2)}$. Precisely, there exists $f: \bR \to [0,1]$ satisfying $\lim_{|c| \to \infty} f(c) = 0$, such that,  uniformly in $n$, for 
$t =\frac{ n( \log n+c)}{2\xi(\sigma^2)}$ we have
\begin{align*}
 \left| \left\|\delta_\Id \cdot  P_\xi^t  -\nu \right\|_{\TV} -\one(c < 0)\right| < f(c). 
\end{align*}
\end{theorem}

The above result establishes a  cut-off phenomenon in total variation for the uniform plane Kac walk, in which the  distance to uniform transitions in a window that is asymptotically small compared to the mixing time.  For general Markov chains, cut-off phenomena were first proved in discrete models such as the nearest neighbor walk on the hypercube $\Z^n/(2\Z)^n$ \cite{Aldous} and the random transposition card shuffling model \cite{DiaconisShahshahani}; see also \cite{cutoff} for an extensive survey. Note that there are also common random walks in which the cut-off phenomenon provably does not occur, for instance, in nearest neighbor random walk on the cycle $\Z/p\Z$.  It is an open and open-ended question  whether the cut-off phenomenon represents a special or generic behavior of Markov chains.  This is partly difficult to answer because models in which the distance to uniform can be accurately analyzed are rare.

The chief simplifying feature of our walk is that the transition kernel is invariant under conjugation, so that the spectrum is completely described by the character theory of $SO(2n+1)$, an observation that goes back to Diaconis and Shahshahani in \cite{DiaconisShahshahani}. Our analysis builds on that of Rosenthal \cite{Rosenthal}, who studied the walk in the special case where $\xi$ is a point mass at a fixed angle $\theta$, proving the lower bound of our theorem in this case, and the  more difficult upper bound for $\theta = \pi$, see also work of Porod on a related walk \cite{Porod}.  
The first stage in our proof of Theorem \ref{L^1 cut-off} is to complete Rosenthal's upper bound analysis for any fixed angle, with strong uniformity in the angle $\theta$.  
\begin{theorem}\label{fixed_theta_theorem}
 Let $\theta = \theta(n)$ vary with $n$ in such a way that $\frac{\log n}{\sqrt{n}} \leq \theta \leq \pi$.  Let $P_\theta$ denote the transition kernel of our walk in the special case when $\xi$ is a point mass at $\theta$.   There exists $f:\bR \to [0,1]$ with $f(c)$ tending to 0 as $|c|\to \infty$ such that, for $t = \frac{n (\log n + c)}{ 2\sigma^2(\theta)}$, uniformly
in $n$ and $\theta(n)$,
\begin{align*}
 \left| \left\| \delta_\Id \cdot P_\theta^t - \nu\right\|_\TV - \one(c < 0) \right| < f(c). 
\end{align*}
\end{theorem}
  As in the special case $\theta = \pi$ treated by Rosenthal, the upper bound of Theorem \ref{fixed_theta_theorem} is obtained by bounding the total variation norm with the $L^2$ norm, which translates the problem into bounding a sum of character ratios, see \cite{PD} for similar computations in
the finite 
group setting.  Although this analysis does not go through for the upper bound in Theorem \ref{L^1 cut-off} because the density need not be in $L^2$,  we are able to sidestep this issue with a truncation argument, in which the extra symmetry of the walk is exploited a second time.

When the measure $\xi$ of Theorem \ref{L^1 cut-off} has support bounded away from zero, the resulting walk converges in $L^2$ and, under the additional  assumption that the support of $\xi$ is bounded away from $\pi$, we are also able to determine the $L^2$ mixing time, and establish a cut-off phenomenon.
\begin{theorem} \label{L^2 cut-off}
Let $\xi$ be a probability measure on $\T_0$ having
support bounded away from 0 and $\pi$, and let $q$
be the point in the support of $\xi$ that is closest to 0. Let
$\sigma(\theta) := \sin\frac{\theta}{2}$ and set $\xi\left(\sigma^2\right) =\int_0^{\pi}
\sigma^2(\theta) \xi(d\theta)$. 
Define for $k \in (0, \infty)$ the threshold
\[
 T_{2,k}(n) := \inf\left\{t \in \zed_{\geq 0}: \left\|\delta_{\Id}\cdot P_\xi^t - \nu\right\|_{L^2} < k\right\}.
\]
We have
\[
 T_{2,k}(n) \sim
\frac{n \log n}{4\sigma^2(q) \wedge 2 \xi(\sigma^2)}.
\]
Furthermore, there exists $0<c_k < \infty$ such that 
\[
 T_{2,k}(n) \leq \max\left( \frac{n \log n}{2 \xi(\sigma^2)}, \frac{n(\log n + 2 \log \log n)}{4 \sigma^2(q)}\right) + c_k n 
\]
With the additional assumption that $\xi$ has positive one-sided lower derivative w.r.t. Lebesgue measure at $q$, we have the lower bound 
\[
 T_{2,k}(n)\geq  \max\left( \frac{n \log n}{2 \xi(\sigma^2)}, \frac{n (\log n- 3 \log \log n)}{4 \sigma^2(q)}\right) - c_k n.
\]
\end{theorem}
 \begin{remark}  When $\sigma^2(q) \leq \frac{1}{2} \xi\left(\sigma^2\right)$ we exhibit a cut-off window of size $O(n\log \log n)$.  It is likely that $O(n)$ is the true size.
\end{remark}

Together,  Theorems  \ref{L^1 cut-off} and \ref{L^2 cut-off} have several surprising features.  For one, the proof of the Theorem \ref{L^2 cut-off} exhibits a competition between the natural representation, which produces the spectral gap, and some moderately sized representations for which the spectrum is smaller, but of higher multiplicity.   We are not familiar with another such naturally occurring walk that has been studied.  Also, for measures $\xi$ such that $\sigma^2(q) < \frac{\xi(\sigma^2)}{2}$ we have a situation in which there is a cut-off phenomenon in both the total variation and $L^2$ norms, the two cut-off points are not the same, and both may be precisely analyzed.  Again, we are not aware of comparable examples.  

We record several further consequences of Theorems \ref{L^1 cut-off} and \ref{L^2 cut-off}.  First, from Theorem \ref{L^1 cut-off} one can also derive near
optimal mixing time results under the Wasserstein distance. Recall that for two
probability measures $\mu,\nu$ defined on a metric space $(\Omega,d)$, the
$L^2$ Wasserstein distance is defined by
\begin{align*}
 W^2(\mu, \nu) = \inf_{(X,Y) \in \cM(\mu,\nu)} \left[\E d(X,Y)^2\right]^{\frac{1}{2}}
\end{align*}
where $\cM(\mu,\nu)$ denotes the set of all couplings of $\mu$ and $\nu$.
 We obtain
\begin{corollary}
Let $\xi$ denote the uniform measure on $\T_0$.  For $t = 2n  (\log n+c)$, the Wasserstein distance of the uniform plane Kac 
walk $\delta_{\Id}\cdot P_\xi^t$ from Haar measure
is bounded by
\begin{align*}
 W^2\left(\delta_{\Id}\cdot P_\xi^t, \nu\right) = o(1), \qquad c \to \infty.
\end{align*}
\end{corollary}
\begin{proof}[Proof sketch]
The optimal transport inequality \cite{OptimalTransport} gives
\begin{align*}
 W^2(\mu,\nu) \ll \sqrt{d^2 \|\mu- \nu\|_\TV}
\end{align*}
where $d$ is the diameter of the state space.  For our walk on $SO(n)$ the
diameter is of order $\sqrt{n}$. It follows in a straightforward way from our
 analysis that $\left\|\delta_{\Id}\cdot P^{2n(\log n+c)} - \nu\right\|_\TV = o\left(\frac{1}{n}\right)$ as $c \to \infty$, since the
second largest eigenvalue is of size $1 - \frac{1}{n} + O(n^{-2})$.
\end{proof}

\begin{remark}
 Oliviera \cite{Oliv} proved that the Kac walk on $SO(2n+1)$ converges in at most $t =
n^2 \log n$ steps under $L^2$ Wasserstein distance. A direct adaptation of his method would give the same upper bound for the uniform
plane Kac walk defined above. Our result does better than this by an order of $n$.  
\end{remark}

 The argument towards Theorem \ref{L^2 cut-off} also implies a cut-off phenomenon at twice the $L^2$ mixing time in the `$L^\infty$ norm with respect to Haar measure'.  Note that in the continuous state space setting there is some subtlety in working with $L^\infty$ norms.  This is explained in Section \ref{L^infty_section}.
\begin{corollary} \label{L^infty cut-off}
The conclusions of Theorem \ref{L^2 cut-off} are valid with the mixing time doubled and the $L^2$ norm replaced by the $L^\infty$ norm with respect to Haar measure.  
\end{corollary}

The main  new technique in our paper is an integral and corresponding differential method of evaluating the character ratios of the orthogonal group at a rotation, which may be of independent interest. Let $\rho_\va$ be an irreducible representation of $SO(2n+1)$ indexed by dominant weight $\va$, of dimension $d_\va$ and character $\chi_\va$.  Rosenthal derives from the Weyl character formula an expression for the ratios \[ r_\va(\theta) = \frac{\chi_\va(R(\theta))}{d_\va}\] as a trigonometric polynomial,  which he could bound successfully when $\theta = \pi$.  We observe that in fact the character ratio is equal to the sum of the residues of a meromorphic function attached to the representation, which  allows us to use a contour integral and the method of stationary phase to give strong estimates for the character, see Section \ref{background_section} for our differential and integral formulas.  It is a surprising feature of the method that the special case $\theta = \pi$ that was resolved by Rosenthal using combinatorial arguments is most difficult for us to handle using the integral approach.  

 While for this investigation we needed to understand the characters of $SO(2n+1)$ evaluated only at a single rotation, in  Appendix A we prove a corresponding multiple integral formula for the characters of $SO(2n+1)$ evaluated at an arbitrary conjugacy class.  Thus, in principle, the techniques developed here could be used to study a random walk generated by elements of $SO(k) \subset SO(2n+1)$ instead of by rotations in a single plane.  A similar formula for character ratios on the symmetric group appears in  \cite{cycle_walk}.

\subsection*{List of Notations}
We collect together symbols used in the course of our argument.
\begin{itemize}
 \item $N= 2n +1$ is the dimension of the ambient space of the special
   orthogonal group $SO(2n+1)$.
   \item $\T_0 = [0,\pi]$.  
 \item $\theta$ is the angle of rotation defining the fixed-angle
   Rosenthal walk in Theorem~\ref{fixed_theta_theorem}. 
 \item $\xi$ is the probability measure on $\T_0$ defining the random angle Rosenthal walk.
 \item $ P_\theta$ or $P_\xi$ is the corresponding Markov kernel.
  \item $\sigma = \sin\frac{\theta}{2}$. 
  \item For $\xi$-integrable $f: \T_0 \to \bR$, $\xi(f) = \int_{\theta \in \T_0} f(\theta)d\xi(\theta).$
 \item $\va \in \N^n$ is the index of an irreducible representation of $SO(2n+1)$. $\vO$ is the 0 string indicating the trivial representation.
\item $\tilde{\va}$ has $\tilde{a}_j = a_j + j - \frac{1}{2}$ is the Rosenthal index for
irreps. $\tilde{\vO}$ is the shifted 0 string.
\item $\balpha$ with $\alpha_j = \frac{\tilde{a}_j}{n}$ denotes the rescale index component. $\bomega$ with $\omega_j = \frac{j-\frac{1}{2}}{n}$ is the rescaled shifted 0 string.
 \item $\vs \in \N^n$, $s_1 = a_1$, $s_i = a_i - a_{i-1}$ for $i > 1$, is the index of an irrep in shift
notation.
% \item $\omega_j = (j-\frac{1}{2})/n$ is the $\alpha_j$ for $\va = \vO$, i.e., the
%trivial representation.
 \item $d_\va$ is the dimension of the irrep $\rho_\va$.
 \item $r_\va(\theta) = \frac{\Tr \rho_\va(R(1,2;\theta))}{ d_\va}$ is the character
ratio at $\va$.
 \item $g_\va(z) = \theta z - \frac{1}{n}\sum_{j =
1}^n
\log(z^2 + \alpha_j^2)$ is the logarithm of the contour
integrand for $r_\va(\theta)$.
% \item $g_\vO(z)= g_\va(z)$ for $\va = \vO$.
% \item $q_n(\theta) = \sin(\frac{\theta}{2}) \sqrt{\frac{n}{\pi}}$ is the
%factor
%in the contour with subexponential growth in $n$. Together $r_\va(\theta) =
%\oint
%q_n(\theta) e^{n g_\va(z)} dz$.
 \item $\omega \sim \cot\frac{\theta}{2}$  is the saddle point of
 $g_\vO$ on the positive real axis.
 \item We write $\oint_\cC f(z)dz$ for the integral $\frac{1}{2\pi i} \int_{z \in \cC} f(z) dz$ of meromorphic function $f$ about contour $\cC$. For real $\omega$, $\oint_{(\omega)} f(z) dz$ indicates integration on the contour $(\omega-i\infty, \omega + i\infty)$.

%\item $a_n < (\omega \vee 1)n$ is the cut-off level for moderate
%representations 
%(see section~\ref{main shift}).
%\item $C_2 (\log n)^{-\frac{1}{2}}$ is the cut-off of $\beta$ in the contour integral.
%  \item $f_\va(z) = \Re \frac{1}{n} \sum_{j=1}^n \log (\alpha_j^2 - z^2)$ is
%the
%essential part of $g_\va(z)$.
%\item $\Delta_k r_\va= e^{-f_{(a_1,\ldots, a_k + 1, a_{k+1} +1, \ldots, a_n
%+1)}(0)}/ e^{-f_\va(0)}$ is the incremental character ratio obtained
%from $\va$ 
%by
%shifting all subsequent indices from $k$.
%\item $\Delta_k r_\va(j)$ is the contribution of $\alpha_j$ to $\Delta_k
%r_\va$; for
%more variants, see Section~\ref{main shift}.
\end{itemize}

\section{Background on representation theory}\label{background_section}

Our results make use of  harmonic analysis on $SO(2n+1)$; a reference
for the theory of harmonic analysis on compact groups is
\cite{hewitt_ross}, see also \cite{serre}  and \cite{edwards}. In addition, Rosenthal's pioneering paper \cite{Rosenthal} on the special case of the present investigation has an excellent coverage of the needed tools from representation theory. Thus our treatment below is somewhat brief.
 
Probability measure $\mu$ has a Fourier development in terms of the
finite dimensional irreducible representations of $SO(2n+1)$.  These are
indexed by weakly increasing integer `highest weights' 
\[\va =(a_1, ..., a_n), \qquad 0 \leq a_1 \leq a_2 \leq ... \leq a_n.\]
We  also frequently use strongly increasing half-integer weights
(Rosenthal's convention, convenient for the Weyl character formula) 
\[\tilde{\va} = (\ta_1, ..., \ta_n), \qquad \ta_k = a_k + k
- \frac{1}{2},\]
normalized weights,
\[
 \balpha = (\alpha_1, ..., \alpha_n), \qquad \alpha_k = \frac{\tilde{a}_k}{n}
\]
and the shift-indices
\[\vs = (s_1, ..., s_n), \qquad s_1 = a_1, \qquad s_i = a_i - a_{i-1},
\quad 1< i \leq n.\] 

Recall a representation $\rho$ of a group $G$ on a finite dimensional (complex) vector space $V$ is a group homomorphism from $G$ to $Hom(V)$, the group of linear transformations on $V$. $\rho$ is irreducible if its image does not lie in a nontrivial split $Hom(V_1) \oplus Hom(V_2)$ with $V_1 \oplus V_2 = V$. 
Every group $G$ has a trivial representation $\rho_0$ that takes each $g \in G$ to $\id \in Hom(0)$, the identity map on the $0$-dimensional vector space. Since $SO(2n+1)$ is defined as the group of orthogonal transformations on $\R^{2n+1}$ with the standard inner product, we get a representation on $\R^{2n+1}$ for free, called the natural representation.

The representation corresponding to $\va = \vO$  is the trivial
representation while $\va$ with $a_n = 1$, 
$a_k = 0,\; k < n$ indicates the natural (matrix) representation.  We
set $d_\va$ for the dimension of irreducible representation
$\rho_\va$, and let $\chi_\va(g) = \Tr(\rho_\va(g))$ be the
character. 

  We  consider measures $\mu$ that are invariant under conjugation.
  Such measures have a Fourier series in terms of the irreducible
  characters on $SO(2n+1)$:\footnote{The characters of $SO(2n+1)$ are real.
    Note that our normalization for $\hat{\mu}(\chi_\va)$ differs from
    Rosenthal's in the factor of $\frac{1}{d_\va}$.} 
\[ \mu \sim \sum_\va \hat{\mu}(\chi_\va) \chi_\va, \qquad
\hat{\mu}(\chi_\va) = \frac{1}{d_\va} \int_{SO(2n+1)}\chi_\va(g) \mu(d
g).\] 
The Fourier map carries convolution to point-wise multiplication:
\[ \widehat{\mu \ast \mu'}(\chi_\va) = \hat{\mu}(\chi_\va)\cdot
\hat{\mu}'(\chi_\va).\]   
Plancherel's identity takes the form
\[ \left\|\mu - \mu'\right\|_{L^2}^2 = \sum_\va d_\va^2 \left|\hat{\mu}(\chi_\va)
  - \hat{\mu}'(\chi_\va)\right|^2,\] with the interpretation that the
right side is finite if and only if $|\mu-\mu'|$ has an $L^2$ density
with respect to Haar measure.  The Fourier coefficients of Haar
measure are given by $\hat{\nu}(\chi_\vO) = 1$, and
$\hat{\nu}(\chi_\va) = 0$ if $\va \neq \vO$. 

In the special case that $\mu_\theta = \delta_{\Id}\cdot
P_\theta$ is the probability measure
generating the fixed-$\theta$ Rosenthal walk, set
$r_\va(\theta) =
\hat{\mu}_\theta(\chi_\va)=\frac{\chi_\va(R_\theta)}{d_\va}
$ for 
the character ratio at rotation $R_\theta$.  
Rosenthal derives the dimension and character ratio formulae
\begin{align}
 d_\va = &\frac{2^n}{1! 3! \ldots (2n-1)!} \prod_{q=1}^n \ta_q \prod_{1
\le s < r \le n} \left(\ta_r^2 - \ta_s^2\right) \label{dimension formula}\\
 r_\va(\theta) = &\frac{(2n-1)!}{\left(2 \sin \frac{\theta}{2}\right)^{2n-1}} \sum_{j=1}^n
\frac{\sin (\ta_j \theta)}{\ta_j \prod_{r \neq j} \left(\ta_r^2 - \ta_j^2\right) }
\label{character ratio}
\end{align}
from the Weyl character formula. We note that his character ratio formula is exactly the sum of the residues at $\{\pm \tilde{a}_j\}_{j=1}^n$ of the meromorphic function
\[
 f_{\va, \theta}(z) = \frac{(2n-1)!}{\left(2 \sin \frac{\theta}{2}\right)^{2n-1}} \frac{ \sin
(\theta z)}{\prod_{j=1}^n \left(\tilde{a}_j^2 - z^2\right)}.
\]
This fact leads to the following integral formula.
\begin{lemma}\label{contour_integral_lemma} Let $\theta \in \T_0$ and let $\va$ index an irrep of $SO(2n+1)$.
 For any $\alpha > 0$,
\begin{equation}\label{contour_integral_formula}
 r_\va(\theta) =  \frac{(2n-1)!}{\left(2n\sin \frac{
\theta}{2}\right)^{2n-1}}\oint_{\Re(z) = \alpha}
\frac{e^{n\theta z}}{\prod_{j = 1}^n (\alpha_j^2 +z^2)}dz.
\end{equation}
\end{lemma}

\begin{proof}
We may express the sum of the residues of $f_{\va,\theta}$ as
\[
 r_\va(\theta) = \oint_{\cR} f_{\va, \theta}(z) dz
\]
where $\cR$ is any rectangle with corners $\pm B \pm in \alpha$, $B > \ta_n$,
oriented counter-clockwise.  As $B \to \infty$, the integral over the
vertical segments go to 0 leaving two horizontal line integrals at $\Im(z) =
\pm in \alpha$. Exchanging
$z \to -z$, the two integrals are seen to be equal, so we keep twice the top
one.  Now split $\sin(\theta z)$ as $\frac{e^{i\theta z}
- e^{-i\theta z}}{2i}$. The contribution from $e^{i\theta z}$ vanishes by
shifting the contour upward to $i \infty$.  For the final result,
replace $z$ by $\frac{iz}{n}$. 
\end{proof}

We also obtain the following differential formula.

\begin{lemma} \label{algebraic lemma}
 Let $\theta \in \T_0$ and let $\rho_\va$, $\va = (a_1, ..., a_n)$ index a representation of $SO(2n+1)$. 
Set $m = a_n$.  Let $\tb_1> \tb_2
> ... > \tb_m$ complement $\{\ta_1, ..., \ta_n\}$ in $\left\{\frac{1}{2}, \frac{3}{2}, ..., m + n -
\frac{1}{2}\right\}$.  The character ratio at $\va$ is given by
\begin{equation}\label{algebraic_formula} r_\va(\theta) = \frac{(2n-1)!}{(2n +
2m -1)! \left(2 \sin \frac{\theta}{2}\right)^{2n-1}} \left(\prod_{k = 1}^m \left(\tb_k^2 +
\partial_\theta^2\right)\right) \left(2 \sin \frac{\theta}{2}\right)^{2n + 2m
-1}.\end{equation}
\end{lemma}

\begin{proof}
Since we know that $\chi_\vO(\theta) = 1$ and this holds uniformly for all
$\theta$ and all orthogonal groups $SO(2m+1)$, we obtain the integral identity
\[1= \frac{(2m-1)!}{\left(2 \sin \frac{\theta}{2}\right)^{2m-1}} \oint \frac{\sin (z
\theta)}{ \prod_{j = 1}^m \left(\left(j-\frac{1}{2}\right)^2 - z^2\right)} dz,\] valid for $m = 0,
1, 2, ...$, $\theta \in \T_0$. 
 
Now taking any contour that encloses the real axis between $\pm \left(m+n - \frac{1}{2}\right)$
we have
\begin{align*}
 r_\va(\theta) &= \frac{(2n-1)!}{\left(2\sin \frac{\theta}{2}\right)^{2n-1}} \oint
\frac{\sin(z\theta)}{\prod_{j =1}^n (\ta_j^2 - z^2)}dz
\\&= \frac{(2n-1)!}{\left(2\sin \frac{\theta}{2}\right)^{2n-1}} \oint \frac{ \left(
\prod_{j =1}^m \left(\tb_j^2 + \partial_\theta^2\right)\right) \sin(z\theta)}{\prod_{j =
1}^{n+m} \left(\left(j-\frac{1}{2}\right)^2 - z^2\right)} dz.
\end{align*}
Passing the differential operator $\prod_{j =1}^m \left(\tb_j^2 +
\partial_\theta^2\right)$ outside the integral, we obtain the required  
expression.
\end{proof}

We also give the differential formula a combinatorial expression in rising powers of
$\sigma = \sin \frac{\theta}{2}$. 
\begin{lemma}\label{combinatorial_char_ratio_lemma}
 Keep the notations of the previous lemma.  We have
\begin{equation}\label{char_ratio_power_series}
 r_\va(\theta) = \sum_{s = 0}^m \frac{(-4)^s(2n-1)!}{(2(n+s)-1)!} E_s
 \sigma^{2s} 
\end{equation}
where 
\[
 E_s = \sum_{ 1\leq j_1 < j_2 < ... < j_s \leq m} \prod_{i = 1}^s
 \left( \left(m+n+i - j_i - \frac{1}{2}\right)^2-\tilde{b}_{j_i}^2
 \right). 
\]
\end{lemma}
\begin{remark}
 Note that $m + n + \frac{1}{2} - j - \tb_j \geq 1$ and is increasing
 in $j$.  In particular, each term in each $E_s$ is positive. 
\end{remark}

\begin{proof}
 We have
\[ \left(\tb_k^2 + \partial_\theta^2\right) \sigma^{2r-1} =
\left(r-\frac{1}{2}\right)(r-1) \sigma^{2r-3} + \left(\tb_k^2 -
\left(r-\frac{1}{2}\right)^2\right) \sigma^{2r-1}.\]  Iterating this in
(\ref{algebraic_formula}) we obtain 
\begin{align}
 \notag &\prod_{k=1}^m \left(\tilde{b}_k^2 + \partial_\theta^2\right)
 \sigma^{2n+2m-1}\\ = &\sum_{S \subseteq [m]} \frac{(2(n+m)-1)!}{(2(n
   + |S|)-1)!}2^{2|S| - 2m} \sigma^{2(n + |S|) -1} \prod_{j\in S}
 \left[\tilde{b}_j^2 - \left(n+m- m_S(j) -\frac{1}{2}\right)^2\right] 
\end{align}
where
$m_S(j) = \#([j-1] \setminus S)$.  The claim now follows on grouping
terms according to $s = |S|$. 
\end{proof}

As an example of the previous two lemmas, we now calculate
the character ratio of several low-dimensional representations. 
These calculations may be used to prove the lower
bound of Theorem \ref{L^1 cut-off}, although as this proof appears in Rosenthal's work, we do not reproduce it here.

\begin{example}\label{char_ratio_example}
The trivial representation $\rho_\vO$ has dimension 1 and
character ratio 1.  

The lowest dimensional non-trivial irrep is
the natural representation $\rho_{(\vO,1)}$.  Its
dimension is $2n+1$.  Using the previous lemma, we may
easily calculate its character ratio.  We have $m = 1$ and
$\tb_1 = n-\frac{1}{2}$ so that $E_1 = 2n$.  Thus
\begin{equation}\label{natural}
 r_{(\vO,1)}(\theta) = 1-
\frac{4\sigma^2}{2n+1}.
\end{equation}

The tensor product $\rho_{(\vO,1)} \otimes \rho_{(\vO,1)}$
decomposes as a direct sum of the trivial representation,
the adjoint square and the symmetric square:
\[
 \rho_{(\vO,1)}\otimes \rho_{(\vO,1)} = \rho_\vO \oplus
\rho_{(\vO,1,1)} \oplus \rho_{(\vO,2)}.
\]
The adjoint square $\rho_{(\vO,1,1)}$ is a representation of
dimension $n(2n+1)$.  For this representation, $m = 1$ and
$\tb_1 = n-\frac{3}{2}$ so that $E_1 = 4n-2$.  Thus its character
ratio is given by
\begin{equation}\label{adjoint_square}
 r_{(\vO,1,1)}(\theta) = 1 -
\frac{4(2n-1)\sigma^2}{n(2n+1)}.
\end{equation}

The symmetric square $\rho_{(\vO,2)}$ is a representation of
dimension $n(2n+3)$.  For this representation, $m = 2$ and
$\tb_1 = n+ \frac{1}{2}$, $\tb_2 = n-\frac{1}{2}$.  Thus $E_1 = 4n+2$ and
$E_2 = 8n^2 + 12n + 4$. Thus the character ratio is given by
\begin{equation}\label{symmetric_square}
 r_{(\vO,2)}(\theta) = 1- \frac{4\sigma^2}{n} + \frac{16
\sigma^4}{n(2n+3)}.
\end{equation}
\end{example}

\section{Deterministic $\theta$: Proof of Theorem \ref{fixed_theta_theorem}}\label{Sketch}
The
starting point for our random walk analysis is the Upper Bound Lemma as discussed in
\cite{Rosenthal}.
\begin{lemma*}[Upper Bound Lemma]\label{upper_bound_lemma}
 Let $\mu$ be a conjugation invariant probability measure on
 $SO(2n+1)$ and let $\nu$ be Haar measure.  The total variation
 distance between $\mu$ and $\nu$ is bounded by 
\[ \|\mu - \nu\|_\TV \leq \frac{1}{2} \|\mu - \nu\|_{L^2} =
\frac{1}{2}\left(\sum_{\va \neq \vO} d_\va^2
  |\hat{\mu}(\chi_\va)|^2\right)^{\frac{1}{2}}.\] 
In particular, the total variation distance to Haar measure for the
Rosenthal walk of 
fixed angle $\theta$ and step $t$ is bounded by
\begin{equation}
 \left\|\delta_{\Id}\cdot P_\theta^t - \nu\right\|_\TV \leq \frac{1}{2}\left(\sum_{\va \neq
\vO} d_\va^2 r_\va(\theta)^{2t}\right)^{\frac{1}{2}}. \label{character sum}
\end{equation}
\end{lemma*}

Recall that for the fixed $\theta$ walk we set $\sigma =
\sin\frac{\theta}{2}$, and that 
for Theorem \ref{fixed_theta_theorem} we need to prove,
uniformly in $n$, that for $t = \frac{n(\log n + c)}{2\sigma^2}$,
$\left\|\delta_{\Id}\cdot P_\theta^{t} - \nu\right\| \to 0$ as $c \to 
\infty$. 
The key estimate that we prove is the following.
\begin{proposition}\label{character_ratio_bound_proposition}
Let $\theta = \theta(n) \in \left[\frac{\log n}{\sqrt{n}},\pi\right]$ and set, as
usual, $\sigma = \sin \frac{\theta}{2}$. 
For all sufficiently large $n$ there exists a fixed constant $C$ independent
of $n$, $\theta$ such that for each non-trivial representation $\rho_\va$, 
\[  \left|r_\va(\theta)\right|^{\frac{n(\log n + C)}{\sigma^2}}<
d_\va^{-2}.\] 
\end{proposition}
We approach this estimate in one of several ways depending upon the size of the representation.  
We say that a representation $\rho_{\va}$ is \emph{small} if $\sum_{i = 1}^n a_i < \frac{n}{\sigma \log n}$.  Those which are not small, but satisfy $a_n \ll \frac{ n}{\sigma}$ are \emph{moderate} and the remaining representations are \emph{large}.  The large representations are further split into ones of \emph{controlled growth} -- those for which there is an index $k = cn$ for which $a_k < Cn$ -- and ones that are \emph{giant}.  For the most sensitive small representations we treat the character ratio using the more exact differential character formula.  For the moderate representations, we estimate the character ratio by treating its integral formula as a perturbation of the integral formula for the trivial character.  For the controlled growth representations we are able to effectively compare the character ratio with another character ratio on a smaller orthogonal group.  For the giant representations, trivial estimates of the integral formula for the character ratio suffice.

We have collected a number of estimates for dimensions of representations in Appendix \ref{dimension_section}, including the proof of the following estimate.
\begin{proposition}\label{sum_of_dimension_bound}
Uniformly in $c \geq 1$ and $n$,
\[ \sum_{\va \neq \vO} d_\va^{\frac{-c}{\log n}} = O\left(e^{-\frac{c}{8}}\right).\]
\end{proposition}
\noindent Combining these Propositions we deduce Theorem
\ref{fixed_theta_theorem}. 
\begin{proof}[Deduction of Theorem
\ref{fixed_theta_theorem}]
The lower bound was proven by Rosenthal.  For the upper
bound set $t = \frac{n(\log n +
c)}{2\sigma^2}$.  By the upper bound lemma,
\[
\left\|\delta_{\Id}\cdot P_\theta^{t} - \nu\right\|_\TV^2 \leq
\frac{1}{4}\sum_{\va \neq \vO} d_\va^2
|r_{\va}(\theta)|^{2t},
\]
and substituting the bounds of Propositions
\ref{character_ratio_bound_proposition} and
\ref{sum_of_dimension_bound}, for $c > C+1$ this is bounded by
\[ \frac{1}{4} \sum_{\va \neq \vO} d_\va^{\frac{(C-c)}{\log n}} =
O\left(e^{\frac{-c}{8}}\right),\] as required.
 
\end{proof}

\subsection{Character ratios of small representations}\label{small_section}

We  first
treat the small characters using an extension of our combinatorial formula, Lemma \ref{combinatorial_char_ratio_lemma} above before moving on to the analysis of the integral for larger representations.
The estimates in this section apply to $\theta$ in the full range $\frac{\log n}{\sqrt{n}} < \theta = \theta(n) < \pi$.
\begin{lemma}\label{small_char_ratio_formula}
 We have the exact evaluation
\begin{equation}\label{E_1_evaluation}
  E_1 = \sum_{i = 1}^n a_i(a_i + 2i-1)
 \end{equation}
 and the bounds
 \[
  |E_1| \leq (m + 2n) \sum a_n, \qquad |E_s| \leq s |E_1| \left[(3m +
    2n) \sum a_n \right]^{s-1}. 
 \]
In particular, for $C> 0$ and for $\va$ satisfying $\sum a_i \leq \frac{Cn}{\sigma
  \log n}$ we have 
 \begin{equation}\label{log_formula_small}
  \log r_\va(\theta) = - \frac{E_1 \sigma^2}{n^2}\left(1 + O\left(\frac{C}{\log n}\right)\right).
 \end{equation}
\end{lemma}

\begin{proof}
Since $\{\ta_i\}_{i = 1}^n$ and $\{\tb_i\}_{i=1}^m$ form a partition
of $\left\{\frac{1}{2}, \frac{3}{2}, ..., m+n-\frac{1}{2}\right\}$ we have 
\[
E_1 = \sum_{i = 1}^m \left(\left(m+n + \frac{1}{2}-i\right)^2 - \tb_i^2\right) = \sum_{i=1}^n
\left(\ta_i^2 - \left(i-\frac{1}{2}\right)^2\right). 
\]
Since $\ta_i = a_i + i-\frac{1}{2}$ the evaluation of $E_1$ follows.

The bound for $|E_1|$ is immediate, since $a_i + 2i - 1 < a_n + 2n = m + 2n$.

To bound $|E_s|$, split off the sum over $j_1$ to
obtain 
\begin{align*}
E_s &= \sum_{j_1=1}^m \left(\left(m + n + \frac{1}{2} - j_1\right)^2
  -\tb_{j_1}^2\right)\\ & \qquad \times \sum_{ j_1 < j_2 <
... < j_s \leq m} \prod_{i =
  2}^s \left(\left(m+n+i - j_i -
    \frac{1}{2}\right)^2-\tilde{b}_{j_i}^2 \right). 
\end{align*}
In the inner sum bound
\[
 \tb_{j_i} + m + n + i - j_i - \frac{1}{2} \leq 3m + 2n,
\]
and 
\[
 \frac{m + n + i - j_i - \frac{1}{2} - \tb_{j_i}}{m + n + \frac{1}{2}
   - j_i - \tb_{j_i}} \leq i. 
\]
which follows from $ m + n + \frac{1}{2} -j - \tb_j \geq 1$.

Therefore
\begin{align*}
 E_s &\leq \sum_{j_1 = 1}^m \left(\left(m + n + \frac{1}{2} -j_1\right)^2 -
 \tb_{j_1}^2\right)(3m + 2n)^{s-1} s!   \sum_{j_1 < ... < j_s \leq m}
 \prod_{i=2}^s \left(m + n +\frac{1}{2} - j_i - \tb_{j_i}\right) \\&\leq
 s\left[(3m + 2n) \sum_{j=1}^m \left(m +n + \frac{1}{2} - j -
   \tb_j\right)\right]^{s-1} \cdot E_1. 
\end{align*}
since \[\sum_{j=1}^m \left(m+n + \frac{1}{2}-j-\tilde{b}_j\right) = \sum_{i=1}^n \left(\tilde{a}_i - i +\frac{1}{2}\right) = \sum a_i.\]
This proves the second bound.

To prove the final claim, substitute the bounds above into
(\ref{char_ratio_power_series}) and use \[\frac{\sigma^2}{n^2} (3m +
2n) \sum a_i = O\left(\frac{C}{\log n}\right).\]  
\end{proof}

\begin{proof}[Proof of Proposition
  \ref{character_ratio_bound_proposition} for small representations] 
When $\sum a_i \leq \frac{n}{\sigma \log n}$, the above lemmas reduce
the proof of Proposition \ref{character_ratio_bound_proposition} to
the estimate (now independent of $\theta$) 
\begin{equation}\label{small_character_target}
 \frac{E_1(\va)}{2n}  >  \frac{\log d_\va}{\log n}.
\end{equation}

This is most convenient to verify in the shift notation $\vs
\leftrightarrow \va$,  $a_i = \sum_{j \leq i}s_j$.  Given a
shift
index $\vs$ write $\vs = \sum_{i=1}^n s_i \ve_i$ for its decomposition
in standard basis vectors.  It is immediate from the expression
(\ref{E_1_evaluation}) for $E_1$ that 
\[ E_1(\vs) \geq \sum_{i =1 }^n s_i E_1(\ve_i).\]  Since Lemma
\ref{shift_ratio_bound} guarantees that $\log d_\vs \leq \sum_{i=1}^n
s_i\log d_{\ve_i}$ we have reduced to checking the claim for $\vs$ of
form $\ve_i$. 

Now
\[ 
 E_1(\ve_i) = \sum_{j = i}^n 2j  = n^2 + n +i - i^2. 
\]
Lemma \ref{s_shift_bound}  proves the estimate
\[
  \log d_{ \ve_i} \leq m\left[1 + \log 2 + \log
\frac{n}{m} + \log
\frac{n}{i-\frac{1}{2}}\right] + 2(n-i)[1 + \log 2] + \log \frac{n}{i-\frac{1}{2}}, 
\]
and thus we have reduced to checking 
\[
 \frac{n^2 + n + i - i^2}{2n} \geq \min(i, n-i) \left[ \frac{\log
     \frac{n}{\min(i, n-i)}}{\log n} + \frac{\log \frac{n}{i}}{\log
     n}\right] + O\left(\frac{n-i}{\log n}\right) 
\]
For $i = \vartheta n$ with $0 < \vartheta \leq \frac{1}{2}$, this reduces to
\[
 \left(\frac{1-\vartheta^2}{2}\right) n + \frac{1 + \vartheta}{2} \geq  \frac{2n}{\log n} \vartheta \log \frac{1}{\vartheta} + O\left( \frac{n}{\log n}\right)
\]
which holds for all $n$ sufficiently large.
For $n-i = \vartheta n$ with $0 < \vartheta \leq \frac{1}{2}$, the corresponding statement is
\[
 \vartheta \left(1 - \frac{\vartheta}{2}\right) n + \left(1 - \frac{\vartheta}{2}\right) \geq \vartheta n \frac{\log \frac{1}{\vartheta(1-\vartheta)}}{\log n} + O\left(\frac{\vartheta n}{\log n}\right), 
\]
which, again, holds uniformly in $\vartheta$ for all $n$ sufficiently large.

\end{proof}

\subsection{Insights from the trivial character}\label{trivial_char_section}
To gain a heuristic understanding of our saddle point arguments for the character ratio of medium and large representations we consider initially the case of
the trivial character $\chi_{\vO}$.  Of course, $r_\vO(\theta)
\equiv 1$; we  analyze $r_\va(\theta)$ for other $\va$ by viewing the
associated integral as a perturbation of the integral for $r_\vO(\theta)$.

In the special case $\va = \vO$, set \[\omega_j = \frac{j-\frac{1}{2}}{n}\] for $\alpha_j$.  We
aim to
choose $\alpha=\omega$ in Lemma \ref{contour_integral_lemma}
where $\omega$ is the location of a real saddle point in the integral.  In
this way, the dominant part of the character ratio is given by the
part of the integral very near $z = \omega$.  To this end, introduce
\begin{equation}\label{def_g_0} g_\vO(z) =   \theta z - \frac{1}{n}\sum_{j =
1}^n
\log\left(z^2 + \omega_j^2\right)
\end{equation}
so that (\ref{contour_integral_formula}) becomes\footnote{The notation
$\int_{(\omega)}$ indicates a contour on the line $\Re(z) = \omega$.}
\[r_\vO(\theta) = \frac{(2n-1)!}{\left(2n \sin \frac{\theta}{2}\right)^{2n-1}} \oint_{(\omega)} e^{n g_\vO(z)} dz.\]

A saddle point in the integral occurs for
$\omega$ solving
\begin{equation}\label{saddle_eqn}g_\vO'(\omega) = \theta  - \frac{1}{n}\sum_{j
= 1}^n
\frac{2\omega}{\omega^2 +
\omega_j^2} = 0.\end{equation} The information that we  need
regarding  $g_\vO$ and its first few derivatives near the saddle point
is contained in the following lemma, whose proof we postpone to the
end of this section. 

\begin{lemma} \label{critical point}
 For $\frac{\log n}{\sqrt{n}} < \theta < \pi - \frac{(\log n)^2}{n}$, the saddle point $\omega$ of $g_\vO = g_{\vO,\theta}$ is
given by
\[\omega = \left(1 +
O\left(\log^{-2}n\right)\right)\cot\frac{\theta}{2}.\] In particular,
\[
 1 + \omega^2 = \frac{1 + O\left(\frac{1}{\log n}\right)}{\sigma^2}.
\]
For fixed $\theta
\in (0, \pi)$,  
\[ \omega = \left(1 + O_\theta\left(n^{-1}\right)\right)\cot\frac{\theta}{2}.\]

For $j \geq 2$ we have $g^{(j)}_{\vO, \theta} = g^{(j)}_{\vO}$ is independent of $\theta$.  Uniformly for $\omega' \in \left[\frac{(\log n)^2}{n}, \frac{\sqrt{n}}{\log n}\right]$ we have the asymptotics

\[g_\vO^{(2)}(\omega')= \frac{2}{1 + {\omega'}^2} \left(1 + O\left(n^{-1} \left(1 \wedge
{\omega'}^{-2}\right)\right)\right), \qquad g_\vO^{(3)}(\omega') \sim \frac{2 \omega'}{\left(1 +
  {\omega'}^2\right)^2},\] 
and for real $t$ satisfying $|t| < \omega' \vee
\frac{1}{2}$, the bound
\[  \left|g_\vO^{(4)}({\omega'} + it)\right|  \leq 40 \left(1 \wedge {\omega'}^{-4}\right)  .\]

In particular, for $\omega$ the saddle point and for all $n$ sufficiently large, \[\Re\left(g_\vO(\omega) -
g_\vO\left(\omega + \frac{i \sqrt{1 + \omega^2}}{3}\right)\right) \geq
\frac{1}{40}.\]
Finally, for any $\omega'>0$ and $t \geq 0$,
 \begin{equation} \label{1/n monotone} \Re g_\vO\left(\omega' + i\left(t +\frac{1}{n}\right)\right)<
   \Re g_\vO(\omega' + it). \end{equation} 

\end{lemma}

Assuming this lemma, it is a standard exercise in the saddle point
method to write 
\[ r_\vO(\theta) = \frac{(2n-1)!e^{n g_\vO(\omega)}}{2\pi \left(2 n
\sin \frac{\theta}{2}\right)^{2n-1}}\left\{\int_{|t|< 10^4\sqrt{\frac{1 + \omega^2}{\log n}}} +
\int_{10^4\sqrt{\frac{1 + \omega^2}{\log n}} < |t| < \frac{\sqrt{1+\omega^2}}{3}} + \int_{|t| > \frac{\sqrt{1+\omega^2}}{3}}\right\} e^{n g_\vO(\omega + it) - ng_\vO(\omega)} dt.\] In the first
integral, we Taylor expand the exponent as
\[-ng^{(2)}_\vO(\omega)\frac{t^2}{2} + i ng^{(3)}_\vO(\omega)\frac{t^3}{6}  +
O\left(n |t|^4 \sup_{|s|< 10^4 \sqrt{\frac{1 + \omega^2}{\log n}}} \left|g^{(4)}_\vO(\omega +
is)\right|\right).\]  Thus the first integral is equal to
\[\left(1 + O\left(\frac{1}{n}\right)\right)\sqrt{\frac{2\pi}{n g_\vO^{(2)}(\omega)}}.\] Near the boundary  $t = \pm \frac{1 \vee \omega}{\log n}$,
the integrand is bounded in size by $e^{-n \log^{-2}n (1 + o(1))}$;
using this 
bound and the fact that the integrand decreases in every period of
length $\frac{1}{n}$  (see \eqref{1/n monotone} above), we may easily
bound the second integral.  For the third, use this 
and note that for $|t| > \omega \vee 1$, the integrand decreases by factors of
$e^{-cn}$ when $t$ doubles.  Thus we may express
\begin{equation}\label{gaussian_approximation} 1 = r_\vO(\theta) =
\frac{(2n-1)!}{\left(2 n \sin \frac{\theta}{2}\right)^{2n-1}}
\frac{e^{ng_\vO(\omega)}}{\sqrt{2\pi n g_\vO^{(2)}(\omega)}} \left(1 +
O\left(\frac{1}{n}\right)\right).\end{equation}

Notice that the same reasoning as above allows us to replace the
integrand with its absolute value. 
\begin{lemma}\label{absolute_value_lemma}
We have 
\begin{equation}\label{absolute_integral}1+O\left(\frac{1}{n}\right) = \frac{(2n-1)!}{\left(2 n\sin
\frac{\theta}{2}\right)^{2n-1}} \frac{1}{2\pi}
\int_{-\infty}^\infty \left|e^{ng_\vO(\omega + it)}\right| dt.\end{equation} 
\end{lemma}
\noindent In view of this lemma we may bound the  character ratios
$|r_\va(\theta)|$  by bounding the real difference  \[\Re(g_\va(\omega
+ it)-g_\vO(\omega + 
it)).\]

\begin{proof}[Proof of Lemma \ref{critical point}]
Write
\begin{align*}
  g_\vO'(z) &= \theta - \frac{1}{n} \sum_{j=1}^n \left[ \frac{1}{z + i
      \frac{j-\frac{1}{2}}{n}} + \frac{1}{z - i \frac{j-\frac{1}{2}}{n}}\right] = \theta + i \sum_{j = 0}^{2n-1} \frac{1}{-inz -n+ \frac{1}{2}+j}
\end{align*}      
This sum may be expressed in terms of the digamma function $\psi(x) = \frac{\Gamma'}{\Gamma}(x)$.  This function satisfies the following properties:
\begin{enumerate}
 \item For integer $k \geq 1$, \[\psi(x + k) - \psi(x) = \sum_{j=0}^{k-1}\frac{1}{x+k}\]
 \item We have $\psi(x) = \psi(1-x) -\pi \cot\left(\pi x\right)$
 \item $\psi$ has expansion uniformly in angular sectors about 0 that omit the negative real axis
 \[
  \psi(x) = \ln(x) - \frac{1}{2x} - \frac{1}{12x^2} + O\left(x^{-3}\right).
 \]
\end{enumerate}
Using the first two properties, we may write
\begin{align*}      
g_\vO'(z)& = \theta + i \left(\psi\left(n + \frac{1}{2} -inz\right) - \psi\left(-n + \frac{1}{2} -inz\right)\right)
\\ &= \theta + i\left(\psi\left(n +\frac{1}{2} -inz\right) - \psi
  \left(n + \frac{1}{2} +inz\right) + \pi \tan\left(i \pi nz\right).
\right) 
\end{align*} 

Suppose
$z=\omega>0$ is real.  Then \[i \pi \tan \left(i\pi n\omega\right) = -\pi +
O(e^{-\pi n\omega}),\] while $\psi(x) = \log x  -\frac{1}{2x} + O\left(x^{-2}\right)$ gives
\begin{align}\notag& i \left(\psi\left(n +\frac{1}{2} -in\omega\right) - \psi
\left(n +
\frac{1}{2} +in\omega \right)\right) 
\\ \label{first_deriv_formula}& \qquad= 2 \tan^{-1}\left(\frac{\omega}{1 +
\frac{1}{2n}}\right) + \frac{1}{n}\frac{\omega}{1 + \omega^2} + O\left(\frac{1 \wedge
\omega^{-2}}{n^2}\right).\end{align} Thus the claim regarding location of the saddle point follows from
\[\frac{\theta}{2} = \cot^{-1}\left(\frac{\omega}{1 + \frac{1}{2n}}\right) +
O\left(\frac{1 \wedge \omega^{-1}}{n}\right) + O\left(e^{-\pi n
\omega}\right).\]

We may express higher derivatives of $g_0$ in terms of the polygammas
$\psi_k(z) = \frac{d^k}{dz^k}\psi(z)$: 
\begin{align*} g_\vO^{(2)}(z) &= n\left(\psi_1\left(n+\frac{1}{2}
      -inz\right)+ \psi_1\left(n + \frac{1}{2} + inz\right) - \pi^2
    \sec^2\left(i\pi n z\right)\right) 
\\ g_\vO^{(3)}(z) &= -in^2\biggl( \psi_2\left(n + \frac{1}{2}
  -inz\right) - \psi_2\left(n+ \frac{1}{2} + inz\right) - 2\pi^3 \tan
\cdot \sec^2\left( i \pi n z\right)\biggr) 
\\ g_\vO^{(4)}(z) &= -n^3\biggl( \psi_3\left(n + \frac{1}{2}
   -inz\right) + \psi_3\left(n+ \frac{1}{2} + inz\right) \\& \qquad - 2
\pi^4(\sec^4 -2 \tan^2\sec^2)(i\pi nz)\biggr) 
.\end{align*}
On $\Re(z) = \omega$, the terms involving $\sec$ are exponentially
small, so may 
be ignored.  All of the claims follow from 
\[\psi_1(z) = z^{-1} + O\left(|z|^{-2}\right), \qquad \psi_2(z) = -z^{-2} +O\left(|z|^{-3}\right),
\qquad \psi_3(z) = 2z^{-3} + O\left(|z|^{-4}\right);\]
 we  just check the explicit bound for $g^{(4)}_\vO$:
\begin{align*}g^{(4)}_\vO({\omega'} + it) &\sim 2\left[ (1 + t - i {\omega'})^{-3} +
(1 -t + i{\omega'})^{-3}\right]\\&=  4\left[ \frac{1 + 3 (t -i{\omega'})^2}{(1-
(t-i{\omega'})^2)^3}\right] = 4\left[\frac{1 + 3t^2 - 3{\omega'}^2
-6it{\omega'}}{(1 - t^2 + {\omega'}^2 +2it{\omega'})^3}\right].\end{align*}
Now $\left|1 - t^2 + {\omega'}^2 +2it{\omega'}\right|^2 = 1 +
2{\omega'}^2 +{\omega'}^4 +t^4 + 2{\omega'}^2t^2 - 2t^2.$    When ${\omega'}\geq
\frac{1}{2}$ and $t \leq {\omega'}$ we have \[\left|1 - t^2 + {\omega'}^2 +2it{\omega'}\right|^2\geq \left(1 + 
{\omega'}^4\right),\] so that
\[ \left|g^{(4)}_\vO({\omega'} + it)\right| \leq 4 \frac{1 + 9 {\omega'}^2}{(1 +
{\omega'}^4)^{\frac{3}{2}}} \leq 40 \left(1 \wedge {\omega'}^{-4}\right).\]  
When ${\omega'} < \frac{1}{2}$, $t \leq \frac{1}{2}$, and therefore $\left|1 - t^2 + {\omega'}^2
+2it{\omega'}\right|^2 \geq \frac{1}{2}$ while 
$|t - i{\omega'}|^2 \leq \frac{1}{2}$.  Thus in this case 
$g_\vO^{(4)} \leq 4 \times \frac{5}{2} \times 2^{\frac{3}{2}}  < 40$.

We may estimate $\Re\left( g_\vO(\omega) - g_\vO\left(\omega + \frac{i \sqrt{1
+\omega^2}}{3}\right)\right)$ by Taylor expansion about $\omega$.

Finally, to prove the last claim, note that
\[ \Re\left( g_\vO\left({\omega'} + i\left(t + \frac{1}{n}\right)\right) - g_\vO({\omega'} + it) \right)= \log
\left|\frac{{\omega'} + i\left(t - \frac{2n-1}{2n}\right)}{{\omega'} +i\left(t +
\frac{2n+1}{2n}\right)}\right| < 0.\]
\end{proof}

\subsection{Character ratios of moderate representations}
\label{moderate_section}
In this section we extend the proof of   Proposition
\ref{character_ratio_bound_proposition} to representations $\rho_\va$
satisfying 
$a_n \leq \frac{2 \cdot 10^6 n}{\sigma}$, $\sigma =
\sin\frac{\theta}{2}$.
Since we have treated the case $\sum a_j \leq
\frac{n}{\sigma \log n}$ in the section on small
representations, we may assume that this no longer holds.  In this
section we  make use of the assumption that $\theta
> \frac{\log n}{\sqrt{n}}$ and we  further assume
that $\pi - \theta \geq \frac{\log^2
  n}{n}$. The 
case in which $\theta$ is closer to $\pi$ is treated in Section
\ref{close_to_pi}.

In Appendix \ref{dimension_section} we prove the following estimate regarding the dimension of moderate representations $\va$. For all sufficiently large $n$, we have 
\[ \exp\left( \frac{n(5-\log \sigma)}{4\cdot 10^6 \log
n}\right) \leq d_\va \leq
\exp\left(\frac{2 \cdot 10^7 n^2}{\sigma}\right) .\]

In analogy with (\ref{def_g_0}) for the trivial character, introduce
(recall $\alpha_j = \frac{\ta_j}{n}$) 
\begin{equation}\label{def_g_a}
 g_\va(z) =   \theta z - \frac{1}{n}\sum_{j =
1}^n
\log\left(z^2 + \alpha_j^2\right),
\end{equation}
so that, fixing the line of integration at the saddle point $\omega\sim
\cot\frac{\theta}{2}$ for the trivial representation,
\begin{align} \label{contour integral for a}
 r_\va(\theta) = \frac{(2n-1)!}{\left(2n \sin\frac{\theta}{2}\right)^{2n-1}}\oint_{(\omega)} e^{n g_\va(z)} dz.
\end{align}

Heuristically, since
\[
 \frac{(2n-1)!}{\left(2n \sin\frac{\theta}{2}\right)^{2n-1}}\frac{1}{2\pi} \int_{-\infty}^\infty \left|e^{n g_{\vO}(\omega + it)}\right| dt = 1 + O\left(\frac{1}{n}\right)
\]
we may bound \[\left|r_\va(\theta)\right| \leq \sup_t \left| \frac{e^{n
      g_\va(\omega + it)}}{e^{n g_\vO(\omega + it)}}\right|.\]  In
practice, we put in the sup bound only for small $|t|\ll
\frac{1}{\sigma \sqrt{\log n}}$ and rely on the rapid decay of the
integrand in $t$ to take care of the rest of the integral.  

We prove the following two estimates.  In the main part of the integral, we prove
\begin{lemma}\label{small_t_extended}
 Let $\va$ satisfy $a_n \leq \frac{2 \cdot 10^6 n}{\sigma}$ and
let $z = \omega + it$ with
$|t| \leq 10^4 \sqrt{\frac{1 + \omega^2}{\log n}}$. Then for sufficiently large $C>0$ for $n$ sufficiently large,
\begin{align} \label{dominant time} \frac{1}{2} n^2\left(\log n + C+1\right)\left(1 +
  \omega^2\right)\cdot \Re\left(g_{\va}(z) - g_{\vO}(z)\right)
\leq - \log d_{\va}.
\end{align}
\end{lemma}

In the tail of the integral we prove
\begin{lemma}\label{integral_tail}
We have the bound
\[ \int_{|t| > 10^4\sqrt{\frac{1 + \omega^2}{\log n}}} \left|e^{n
g_\va(\omega + it)}\right|
dt \leq \exp\left(n g_\vO(\omega) - \frac{(10^8+o(1))n}{\log
n} \right). \]
\end{lemma}

We give the short deduction of Proposition \ref{character_ratio_bound_proposition} for moderate representations.
\begin{proof}[Proof of Proposition \ref{character_ratio_bound_proposition}.]

Recall that 
(\ref{absolute_integral}) of Section \ref{trivial_char_section} gives
\[1+O\left(\frac{1}{n}\right) = \frac{(2n-1)!}{\left(2n \sin
\frac{\theta}{2}\right)^{2n-1}} \frac{1}{2\pi}
\int_{-\infty}^\infty \left|e^{ng_\vO(\omega + it)}\right| dt.\]
Thus we have the bound
\begin{align}
\label{two_part_bound} \left|r_\va(\theta)\right|\left(1 - O\left(\frac{1}{n}\right)\right) &\leq \sup_{|t| \leq
10^4\sqrt{\frac{1 +
\omega^2}{\log n}}}
\exp\left(n\Re\left(g_\va(\omega + it) - g_\vO(\omega +it)\right)\right) \\& \notag 
+ \sqrt{2\pi g_0^{(2)}(\omega) n} \int_{|t|>10^
4\sqrt{\frac{1 +
\omega^2}{\log n}}} \exp\left(n\Re\left(g_\va(\omega + it) -
g_\vO(\omega)\right)\right) dt.
\end{align}
By Lemmas \ref{small_t_extended} and \ref{integral_tail} the RHS is bounded by
\[
 \leq \exp\left( -\frac{2 \sigma^2 \log
   d_\va}{n (\log n + O(1))}\right) + O\left(\exp\left(\frac{-10^8 n (1 +
o(1))}{\log
       n}\right)\right). 
\]

Since $d_\va \leq \exp\left(\frac{2 \cdot 10^7 n^2}{\sigma}\right)$,
the last expression is, for some $c > 0$,
\[\leq \left(1 + O\left(\exp\left(\frac{-cn}{\log n}\right)\right)\right) \exp\left( -\frac{2 \sigma^2\log d_\va}{n (\log n + O(1))}
  \right) .\] 
We have thus shown that
\[
 \log \left|r_\va(\theta)\right| \leq \frac{-2\sigma^2}{n \log n} \log d_\va \left(1 +
 O\left(\frac{1}{\log n}\right)\right) + O\left(\frac{1}{n}\right). 
\]
Since $d_\va \geq \exp\left(\frac{n(5 - \log
\sigma)}{4\cdot 10^6\log
    n}\right)$ and $\sigma \geq \frac{\log n}{\sqrt{n}}$ it
follows
that in fact 
\[
 \log \left|r_\va(\theta)\right| \leq \frac{-2\sigma^2}{n \log n} \log d_\va \left(1 +
 O\left(\frac{1}{\log n}\right)\right),  
\]
which proves Proposition \ref{character_ratio_bound_proposition} for
moderate representations. 
\end{proof}

For our two integral estimates it is convenient to consider the length function
\begin{equation}
\label{distance_product}\ell(x;t, \omega)^2 = \left|\omega + it -
  ix\right|^2 \left|\omega + it + ix\right|^2 = \left(x^2 +\omega^2 - t^2\right)^2 + 4\omega^2 t^2.
\end{equation}
This function plays a prominent role in the next two sections because we may write 
\[
 \Re\left[g_\va(\omega + it) - g_\vO(\omega + it)\right] = \frac{1}{n}
 \sum_{j =1}^n (\log \ell(\omega_j;t,\omega) - \log \ell(\alpha_j; t,
 \omega)). 
\]
We record several of its simple properties.
\begin{lemma}\label{ell_lemma}
Let $t$ and $\omega >0$ be fixed and consider $\ell$ as a function of
$x$ only.  If $|t| < \omega$ then $\ell$ is minimized at $x = 0$ with
minimum $t^2 + \omega^2$.  If $|t| \geq \omega$ then $\ell$ is
minimized at $|x| = \sqrt{t^2 - \omega^2}$ with minimum $2 |t|\omega$.
In the case $|t| > \omega$, for $0 < \delta < \sqrt{t^2 - \omega^2}$
we have  
\[
\ell\left(\sqrt{t^2 - \omega^2} - \delta\right) < \ell\left(\sqrt{t^2 -\omega^2} + \delta\right).
\]
\end{lemma}

\subsubsection{Estimates in the bulk}

 As in the small representation section, our argument
is based upon making shifts to the index, but whereas in
that section
we shifted the rightmost indices first, here we shift indices from
left to right.  
The main estimate which we prove for the integrand is as follows.
\begin{lemma}\label{increment_upper_bound}
 Let $\vs \in \N^{j}$ and let $(\vs, \vO_{n-j})$ index a representation in shift notation.  Recall that we set $\ve_j$ for the $j$th standard
unit vector. For $t$ satisfying $t \leq 10^4 \sqrt{\frac{1 +
\omega^2}{\log n}}$,
\begin{align}
 \notag & \Re\left[ g_{(\vs,  \vO)+\ve_j}(\omega + it) - g_{(\vs,
\vO)}(\omega + it)\right]\leq
 \\\label{left_shift_lower_bound}&\quad  -\frac{1}{2n} \frac{\left(1 -
\frac{j-1}{n}\right)\left(1 + \frac{j + 2|\vs| }{n}\right)\left( \left(1 + \frac{|\vs| }{n}\right)^2 +
\left(\frac{j +|\vs|  }{n}\right)^2 + 2\omega^2\right)}{\left(\left(1 + \frac{|\vs| }{n}\right)^2 +
\omega^2\right)^2}\\
&\qquad\times\left(1 - \frac{4 t^2}{1 + \omega^2} + O\left(\frac{1}{n}\right)\right). \notag
\end{align}
\end{lemma}

\begin{remark}
 For reference, note that if $j\approx n$ and $|\vs|$ is much smaller
 than $n\left(1 + \omega^2\right)$ then  the RHS of the above bound is roughly
 $-\frac{2 \left(1- \frac{j}{n}\right)}{n\left(1 + \omega^2\right)}$. 
\end{remark}

\begin{proof}
 Let $\va$ correspond to $(\vs,  0)$.  Then $\alpha_j = \frac{j + |\vs|
-\frac{1}{2}}{n}$ and $\alpha_n = 1 + \frac{|\vs|  - \frac{1}{2}}{n}$.  We have
\begin{align*}
 &\Re\left[ g_{(\vs,  \vO)+\ve_j}(\omega + it) - g_{(\vs,  \vO)}(\omega +
it)\right] = \frac{1}{n}\log \frac{\ell(\alpha_j)}{\ell(\alpha_n +
\frac{1}{n})} 
\\ &\qquad= \frac{1}{2n} \log \frac{\left(\alpha_j^2 + \omega^2 - t^2\right)^2 +
  4\omega^2 t^2}{\left(\left(\alpha_n + \frac{1}{n}\right)^2 + \omega^2 - t^2\right)^2 + 4\omega^2
  t^2} 
\\ &\qquad= \frac{1}{2n}\log\left[1 - \frac{\left(\left(\alpha_n +\frac{1}{n}\right)^2 -
    \alpha_j^2\right)\left(\left(\alpha_n + \frac{1}{n}\right)^2 + \alpha_j^2 + 2\omega^2 -
    2t^2\right)}{\left(\left(\alpha_n + \frac{1}{n}\right)^2 + \omega^2 -t^2\right)^2 + 4
    \omega^2t^2}\right] 
\\&\qquad \leq - \frac{1}{2n}\frac{\left(\left(\alpha_n +\frac{1}{n}\right)^2 -
  \alpha_j^2\right)\left(\left(\alpha_n + \frac{1}{n}\right)^2 + \alpha_j^2 + 2\omega^2 -
  2t^2\right)}{\left(\left(\alpha_n + \frac{1}{n}\right)^2 + \omega^2 +t^2\right)^2 } 
\end{align*}
Since $\alpha_n + \frac{1}{n} \geq 1$, the result follows on substituting the
values of $\alpha_j$ and $\alpha_n$ and factoring out expressions
involving $t$ (use $\frac{t^2}{1+\omega^2} = \delta$ and  $\frac{1-2\delta}{(1 + \delta)^2} > 1-4\delta$). 
\end{proof}

The remainder of the proof of Lemma \ref{small_t_extended} is concerned with bounding the growth of the dimension.  This is the most sensitive part of the shifting argument, so we prove an initial estimate first which holds only for $a_n \ll \frac{n}{\log n}$ in full generality, although it covers all moderate representations if $\theta$ is not too close to $\pi$.  For convenience, we quote the dimension bound which we use, proved in Lemma \ref{left_shift_dim_bound} of Appendix \ref{dimension_section}.
\begin{lemma*}
 Let $1 \leq j \leq n$ and  $\vs \in \N^{j}$.  Let $m =
\min(j, n-j+1)$ and let $1 \leq \eta \leq m$ be a parameter. Write
\[|\vs|_{\eta, \loc} = \sum_{j - \eta \leq i \leq j} s_i.\]
 We have the bound
\begin{align*}
 \log \frac{d_{(\vs,\vO)+ \ve_j}}{d_{(\vs, \vO)}}& \leq m \left[ \log \frac{n
+ |\vs|_{\eta,\loc}}{m + |\vs|_{\eta,\loc}}  + \log \frac{n + j}{m+j} +2 \right]\\&
\qquad + \eta  \log (n-j+\eta) + 2(n-j+1) + \log\frac{n}{j}+ O(1) .
\end{align*}

\end{lemma*}
Our initial lemma is as follows.

\begin{lemma}\label{comparison_lemma}
 Let $(\vs,  \vO)$, $\vs \in\N^{j}$, be the shift index of
$\va$,  and let $z =\omega + it$ with
   $|t| 
\leq 10^4 \sqrt{\frac{1 + \omega^2}{\log n}}$.   

\begin{enumerate} 
 \item[(i)] If  $a_n =|\vs|  \leq \frac{n\sqrt{1 +
\omega^2}}{ \log n}$ then for sufficiently large fixed $C$, 
\begin{equation}\label{increment_comparison} \frac{1}{2} n^2(\log n + C)\left(1 +
\omega^2\right)\cdot \Re\left[g_{(\vs,  \vO)+ \ve_j}(z) - g_{(\vs,
\vO)}(z)\right] \leq - \log \frac{d_{(\vs, 
\vO)+\ve_j}}{d_{(\vs,
\vO)}}.\end{equation}
\item[(ii)]  In the range  $\frac{n\sqrt {1 + \omega^2}}{ \log n}
 \leq |\vs|
\leq \frac{2 \cdot 10^6 n}{\sigma} $ there is a fixed
constant $C'$ such that
if $\omega > C'$ then the bound
(\ref{increment_comparison}) continues to hold.
\item[(iii)]   For $\omega < C'$ there
is a third fixed constant $C''>0$ such that
\[\frac{1}{2} n^2(\log n + C)(1 + \omega^2)\cdot \Re\left[g_{(\vs,
\vO)+\ve_j}(z) - g_{(\vs, \vO)}(z)\right] \leq -
C''
(n-j+1 )\log n.\]
\end{enumerate}
\noindent 
Each of $C, C'$ and $C''$ is independent of $n, \theta,$ and $\vs$.  In
particular, 
\begin{enumerate}
\item[(iv)] The bound (\ref{increment_comparison}) holds
unless $n-j+1 \leq n^{\kappa}$
for a fixed universal $\kappa < 1$.
\end{enumerate}
\end{lemma}

\begin{proof}
(i) The restrictions on $|\vs|$ and $t$ allow us to write
(\ref{left_shift_lower_bound}) as 
\[
 \Re\left[ g_{(\vs,  \vO)+\ve_j}(z) - g_{(\vs,
\vO)}(z)\right]\leq -\frac{\left(1 - \frac{j-1}{n}\right)\left(1 +
\frac{j}{n}\right)\left(1 + \frac{j^2}{n^2} + 2\omega^2\right)}{2n\left(1 +
\omega^2\right)^2}\left(1 + O\left(\frac{1}{\log n}\right)\right).
\]
The relative error term may evidently be ignored by choosing
the constant $C$
sufficiently large. 
Set $m = n-j + 1$.  For all $j$  the LHS above is less
than 
\[-\frac{2 }{(1 + \omega^2)n^2 \log n}\cdot m \log n  \cdot \max\left(\left(1 -
O\left(\frac{m}{n}\right)\right),c \right)\] for some fixed $c > 0$.   
Meanwhile,   we have a bound
of $\geq -m\log\left(\frac{n}{m}\right) + O(m)$ for the RHS of
(\ref{increment_comparison}) by taking $\eta = 1$ in Lemma 
\ref{left_shift_dim_bound}. The first claim follows by choosing $C$
sufficiently large to cover the case of small $m$.

(ii)  For $|\vs| \leq \frac{2 \cdot 10^6 n}{\sigma}$ if $\omega$ is larger than a
sufficiently large fixed constant %\footnote{1000000 will
%do.}
then in fact the net effect of the $\frac{|\vs|}{n}$ terms
in the RHS of
(\ref{left_shift_lower_bound}) is negative, so that the above argument goes
through without further restriction on $|\vs|$.  

(iii)  In any case, in the range $|\vs|  \leq \frac{2 \cdot 10^6 n}{\sigma}$, inclusion of 
the factors of $\frac{|\vs| }{n}$ changes the RHS of
(\ref{left_shift_lower_bound}) by at most a constant factor,
which proves the third claim.

(iv) The claim regarding $\kappa$ now follows, since uniformly in
$|\vs| \leq \frac{2 \cdot 10^6 n}{\sigma}$ we have that the LHS of (\ref{increment_comparison}) is
less than $-c' m \log n$, for a fixed $c' > 0$, while the RHS is $\geq -m
\log\left(\frac{n}{m}\right) + O(m)$, with $m = n-j+1$ as before.
\end{proof}

Lemma \ref{comparison_lemma}  allow us to prove Proposition
\ref{character_ratio_bound_proposition} for moderate representations
that satisfy $a_n \ll \frac{n}{\sigma \log n}$.  While the strict
increment inequality (\ref{increment_comparison}) does not necessarily
hold for every increment in the range  $a_n \ll \frac{n}{\sigma}$, we are able to complete the proof of Lemma \ref{small_t_extended} by proving that this estimate holds on average when nearby shifts are moved together.

\begin{proof}[Proof of Lemma \ref{small_t_extended}]
Let $\vs$ correspond to $\va$ and set $|\vs| = a_n = m$.
 Using standard basis vectors we may write (in accordance with
 left-to-right shift) 
\[\vs = \sum_{j = 1}^m \ve_{i_j}, \qquad i_1 \leq i_2 \leq ... \leq i_m;\] set
also $\vs^j = \sum_{k = 1}^j \ve_{i_k}$ with $\vs^0 = \vO$.  Obviously
\begin{align} \notag
 &\frac{1}{2} n^2(\log n + C+1)(1 + \omega^2)\cdot
 \Re\left[g_{\va}(z) 
- g_{\vO}(z)\right] + \log d_{\va}
\\\label{sum_stages} &\quad = \sum_{j = 1}^m \left[ \frac{1}{2} n^2(\log n +
C+1)(1 + \omega^2)\cdot \Re\left[g_{\vs^j}(z)
- g_{\vs^{j-1}}(z)\right] + \log
\frac{d_{\vs^j}}{d_{\vs^{j-1}}}
\right].
\end{align}
If either $m<\frac{n\sqrt{1 + \omega^2}}{ \log n}$ or $\omega > C'$
then we may apply either (i) or (ii) of the previous lemma to conclude
that each 
term in the sum is negative so that we are done.  So we may assume
that $m$ is large 
and that $\omega$ is bounded.  

Call $k = \left\lfloor\frac{n\sqrt{1 +
\omega^2}}{ \log n}\right \rfloor$ and let $m_0 > k$ denote the index
of the first 
positive term in the sum. Note that $i_{m_0} \geq n- n^\kappa$ for some fixed
$\kappa < 1$, by (iv) of Lemma \ref{comparison_lemma}.

We first argue that we may assume that $m - m_0$ is large by noting
that the first $k$ terms in the sum of (\ref{sum_stages}) are
substantially negative. Recall that $\omega$ is assumed to be bounded.
Applying Lemma \ref{dimension_lower_bound} with 
$\eta = n - i_k+1$, we deduce for $n$ sufficiently large that
\[ d_{\vs^k} \geq \exp\left(\frac{k(n-i_k+1)}{3}\right) \geq \exp\left(
\frac{n(n-i_k+1)}{3 \log n}\right).\]  Since the sum up to $k$ in
(\ref{sum_stages}) is negative, even when $C + 1$ is replaced by $C$,  it
follows by comparing with $\frac{\log d_{\vs^k}}{\log n}$ that
\begin{align}
&\frac{1}{2} n^2(\log n + C+1)(1 + \omega^2)\cdot
\Re\left[g_{\vs^k}(z) - g_{\vO}(z)\right]+  \log
d_{\vs^k}\nonumber\\
 &\qquad \leq (1 +
o(1))\frac{-n(n-i_k+1)}{3 (\log n)^2}. \label{C+1 to C} 
\end{align}

  Now for $j > k$, Lemma
\ref{left_shift_dim_bound} gives that
\begin{equation}\label{uniform_upper_bound_dim}\log
\frac{d_{\vs^j}}{d_{\vs^{j-1}}} \leq  (n-i_j + 1)(\log n + O(1))\leq
(n-i_k +
1)(\log n + O(1)),\end{equation}
 and so we immediately obtain that (\ref{sum_stages}) is
negative unless $m - m_0 \geq \frac{n}{4 (\log n)^3}.$

Let $\delta > 0$ be a small, fixed, positive constant and set \[J = \left\lceil
\frac{ \log (n-i_{m_0}+1)}{\log (1 + \delta)}\right \rceil \ll \log n.\]   We
partition $\{i_{m_0}, i_{m_0 + 1}, ..., i_m\}$ into $J$ sets by defining
\[ S_j = \{i_\lambda: i_\lambda \in n- I_j\}, \qquad I_j = \left[\left(1 +
\delta\right)^{j-1},
\left(1 + \delta\right)^j\right), \qquad j=1, 2, ..., J.\]
We perform a trimming on the sets $S_j$.
Let $M = \frac{2n}{3 J (\log n)^4} \gg \frac{n}{(\log n)^5}$. We discard all
 $S_j$ with $|S_j| < M$.  From each remaining set we form $S_j'$
by removing the smallest $\frac{M}{2}$ of the $i_\lambda$ from $S_j$.  Altogether we
have discarded at most $\frac{3}{2} JM \leq \frac{n}{(\log n)^4}$ of the
$i_\lambda$. Now in view of (\ref{uniform_upper_bound_dim}), the total
contribution to (\ref{sum_stages}) of the discarded $i_\lambda$ is bounded by
\[(1 +
o(1))(n-i_k + 1) \frac{n}{(\log n)^3},\] which is negligible; see
\eqref{C+1 to C}. 

We now claim that if $\delta$ was chosen to be appropriately small, then
for each remaining $i_\lambda \in S_j'$,
\[\frac{1}{2} n^2(\log n +
C+1)(1 + \omega^2)\cdot \Re\left[g_{\vs^\lambda}(z)
- g_{\vs^{\lambda-1}}(z)\right] \leq - \log
\frac{d_{\vs^\lambda}}{d_{\vs^{\lambda-1}}} .\]  Setting $\vs = \vs^{\lambda -1}$, $j
= i_\lambda$ and $\eta = 3\delta (n- i_\lambda+1)$ in Lemma
\ref{left_shift_dim_bound} we find that $|\vs|_{\eta, \loc} \geq \frac{M}{2}$, (it includes all
of the deleted points from the set containing $i_\lambda$)
\[\log \frac{d_{\vs^\lambda}}{d_{\vs^{\lambda-1}}} \leq 3\delta (n-i_k + 1) \log
n + O((n-i_k + 1) \log\log n).\] Choosing $3\delta$ sufficiently smaller than $
C''$ from the previous lemma proves the claim and finishes the proof.
\end{proof}

\subsubsection{Integral tail estimate}
Lemma \ref{small_t_extended} completes our supremum bound
for small
$t$, so we now turn to bounding the tail of the integral in
(\ref{contour integral for a}). 
\begin{lemma}\label{all_reps_bound}
Introduce the function
\begin{align}\label{def_m_theta}
  m_\omega(t) &= \max_{\va: a_n \leq \frac{2\cdot 10^6 n}{\sigma}}
 \Re\left(g_\va\left(\omega + it\right)\right).
\end{align}
This satisfies the following properties.
\begin{enumerate}
 \item The maximum in $m_\omega(t)$ is achieved at $\va = \vO$ when $t^2 \leq
\omega^2 + \frac{1}{4}$.
\item The monotonicity property $m_\omega(t) \geq m_\omega\left(t + \frac{1}{n}\right)$
holds for all $t \geq 0$.
\item For $t > 4 \cdot 10^6 \sqrt{1 + \omega^2}$, there is a
$c > 0$ such
that
$m_\omega(2t)
\leq m_\omega(t) - cn.$
\end{enumerate}

\end{lemma}
\begin{proof}
Recall \[\Re\left(g_\va(\omega + it)\right) = \omega \theta - \frac{1}{n}\sum_j
\log (\ell(\alpha_j)).\] 
with
\[
 \ell(x; \omega, t) = \left(x^2 +\omega^2 - t^2\right)^2 + 4\omega^2 t^2.
 \]
 Set $c =  \sqrt{\max(0, t^2 - \omega^2)}$.  Since in  $x > 0$ we have $\ell(x; \omega, t)$ is decreasing in $|x - c|$,  the optimal $\va$ has $\alpha_j$ forming a continuous block about $c$, that is, is of the form $(k)^n$ for some $k$. For (1), by the last claim of Lemma \ref{ell_lemma} for $0 \leq \delta \leq c$, $\ell(c-\delta) < \ell(c + \delta)$, so that the optimal choice is a continuous block including 0.

The second claim holds, since if $(k)^n$, $k \geq 1$
achieves the maximum in  
$m_\omega\left(t + \frac{1}{n}\right)$ then 
\[m_\omega(t) \geq \Re\left(g_{(k-1)^n}(\omega + it)\right) >
\Re\left(g_{(k)^n}\left(\omega + i\left(t + \frac{1}{n}\right)\right)\right) = m_\omega\left(t + \frac{1}{n}\right),\] 
while if $\vO$ achieves the maximum in $m_\omega\left(t + \frac{1}{n}\right)$ then 
\[m_\omega(t) \geq \Re\left(g_{\vO}(\omega + it)\right) > \Re\left(g_{\vO}\left(\omega +
i\left(t + \frac{1}{n}\right)\right)\right) = m_\omega\left(t + \frac{1}{n}\right),\] 
by applying
the last claim of Lemma \ref{critical point}.

Finally, notice that the
restriction on $a_n$ is equivalent to $\alpha_n \leq 1 + 
(2 \cdot 10^6+o(1)) \sqrt{1 + \omega^2}$.  For all $|t| > 4 \cdot
10^6  \sqrt{1 +
\omega^2}$ the maximizing $\va$ is easily seen to be the
block with $a_n$ as large as possible.  The last claim now
follows from Euclidean geometry.  
\end{proof}

We now bound the tail of the
integral for $r_\va(\theta)$.

\begin{proof}[Proof of Lemma \ref{integral_tail}]
Recall that this is the claim
\[ \int_{|t| > 10^4\sqrt{\frac{1 + \omega^2}{\log n}}} \left|e^{n
g_\va(\omega + it)}\right|
dt \leq \exp\left(n g_\vO(\omega) - \frac{(10^8+o(1))n}{\log
n} \right). \]

 In the integral we may evidently replace $g_\va(\omega + it)$ with
$m_\omega(t)$. For $|t| = 10^4\sqrt{\frac{1 + \omega^2}{\log n}} +
O\left(\frac{1}{n}\right)$, for sufficiently large $n$, $m_\omega(t) =
\Re(g_\vO(\omega+
it))$ and Taylor
expansion of
$g_\vO$ around $\omega$ gives, see Lemma \ref{critical point},
\[\Re g_\vO(\omega + it) \leq g_\vO(\omega) - \frac{10^8 +
o(1)}{\log n}.\] 
The
bound now follows easily on applying the monotonicity of $m_\omega$ and rapid
decay for $|t| > 4\cdot 10^6 \sqrt{1 + \omega^2}$.
\end{proof}

\subsection{Large representations}\label{large_section}
Among those large representations, for which $a_n >
\frac{2 \cdot 10^6 n}{\sigma}$,
we distinguish further two kinds.  Let $k = n - \left\lfloor \frac{\sigma n}{2}\right\rfloor$.  If $a_k < \frac{2\cdot 10^6 k }{\sigma}$ we
say that
$\rho_\va$ has `controlled growth.'  Otherwise we say that $\rho_\va$ is
`giant.'

\subsubsection{Controlled growth} In the case that $\rho_\va$ has controlled
growth, let $k< m < n$ be maximal
such that $a_m < \frac{2\cdot 10^6 m}{\sigma}$.  We are
going to view
$r_\va(\theta)$ as a perturbation of the character ratio
$r_{\va(m)}(\theta)$ on 
$SO(2m+1)$, where $\va(m)$ denotes $\va$ truncated at $a_m$.

Let $\omega'$ denote the saddle point for $g_{\vO(m)}$, solving
$
 g'_{\vO(m)}(\omega') = 0
$
for $SO(2m+1)$, and set $\omega^* = \frac{m}{n} \omega'$.

We  choose a 
contour in the integral formula for $r_\va(\theta)$
passing through $\omega^*$ and given by $\cC = \cC_1 \cup \cC_2$ where
\[
 \cC_1 = \left\{z = \omega^* + it: |t|\leq \frac{m}{n} 10^4
\sqrt{\frac{1
+{\omega'}^2}{\log m}}\right\}
\]
and where $\cC_2$ depends upon $\omega^*$,
\begin{align*}
\omega^* > \frac{1}{400}: \qquad\cC_2 &=  \left\{z = \omega^* + it: |t|>  10^4 \frac{m}{n}
\sqrt{\frac{1
+{\omega'}^2}{\log m}}\right\}, \\
\omega^* \leq \frac{1}{400}: \qquad \cC_2 &=
\left\{\omega^* + it:  10^4\frac{m}{n}\sqrt{\frac{1+{\omega'}^2}{\log n}}
\leq
|t|
\leq \frac{m}{n}\frac{\sqrt{1 + {\omega'}^2}}{3}\right\}\\&\qquad \cup
\left\{\alpha \pm
\frac{m}{n}\frac{i\sqrt{1 + {\omega'}^2}}{3}: \alpha \in
\left[\omega^*, \frac{1}{400}\right]\right\} \\&
\qquad \cup
\left\{\frac{1}{400}+ it: |t| > \frac{m}{n}\frac{\sqrt{1 +
{\omega'}^2}}{3} \right\}.
\end{align*}

For the integrand on $\cC_1$ we prove the following estimate.
\begin{lemma}\label{controlled_small_t}
 For $|t| < 10^4 \frac{m}{n}\sqrt{\frac{1 + \omega^2}{\log m}}$ and for
a
sufficiently large
fixed constant $C$,
\begin{align} \label{controlled growth small t}
 \Re \left( ng_\va(\omega^* + it) - mg_{\vO(m)}\left(\omega' + \frac{n}{m} it\right)\right)
-(n-m) \log (\sigma^2)\leq  -\frac{2 \sigma^2
\log d_\va }{ n (\log n+C)}. 
\end{align}
\end{lemma}

For the integral on $\cC_2$ we prove
\begin{lemma}\label{C_integral_bound}
 There is $c > 0$ such that
\[\frac{\sqrt{2 \pi m g^{(2)}_\vO(\omega')}}{\left(\sin \frac{\theta}{2}\right)^{2(n-m)}}\int_{z \in \mathscr{C}_2} \exp\left(n\Re
g_\va(z) - mg_\vO(\omega')\right) d|z|  \leq \exp\left(\frac{-cn}{\log n} -\frac{2 \sigma^2
\log d_\va }{ n (\log n+C)}\right).\]
\end{lemma}

We give the short deduction of Proposition \ref{character_ratio_bound_proposition} for controlled-growth representations.
\begin{proof}[Proof of Proposition
\ref{character_ratio_bound_proposition} ]
Note that, as we are comparing characters on $SO(2n+1)$ and $SO(2m+1)$, the leading factors in the integral representation have ratio
\[
 \frac{(2n-1)!}{(2m-1)!}\frac{\left(2m\sin\frac{\theta}{2}\right)^{2m-1}}{\left(2n\sin \frac{\theta}{2}\right)^{2n-1}} < \frac{1}{\left(\sin \frac{\theta}{2} \right)^{2(n-m)}}.
\]
Arguing as in the proof for  moderate
representations,
\begin{align}\notag
&\left(1 - O\left(\frac{1}{n}\right)\right)\left|r_\va(\theta)\right| \\&\leq \frac{1}{\left(\sin \frac{\theta}{2}\right)^{2(n-m)}} \sup_{|t|\leq 10^4 \frac{m}{n}
\sqrt{\frac{1
+\omega^2}{\log m}}} \exp\left(\Re \left(ng_\va(\omega^* + it) -
mg_\vO\left(\omega' +
\frac{n}{m} it\right)\right)\right) \nonumber\\
& \qquad + \frac{\sqrt{2 \pi m g^{(2)}_\vO(\omega')}}{\left(\sin \frac{\theta}{2}\right)^{2(n-m)}}\int_{z \in \mathscr{C}_2} \exp\left(n\Re
g_\va(z) - mg_\vO(\omega')\right) d|z|  \notag
\end{align}
Putting  in the bounds of Lemmas
\ref{controlled_small_t} and \ref{C_integral_bound}, we obtain
\[
 \log |r_\va(\theta)| \leq O\left(\frac{1}{n}\right) - \frac{2 \sigma^2
\log d_{\va} }{ n (\log n+C)}.
\]
The error term of size $O\left(\frac{1}{n}\right)$ is dealt with as for the moderate representations.
\end{proof}

\subsubsection{Proof of estimates}
Throughout this
section we argue by keeping track of the
incremental change to the integrand
$e^{n g_{\va(m)}(z)}$ and to $d_{\va(m)}$ as we append
successively each
$a_j$, $j > m$.  In Appendix \ref{dimension_section} we write the dimension formula incrementally as
\begin{align*}
 d_\va &= \prod_{k=1}^n d_\va(k), \qquad d_\va(k) = \frac{\ta_k}{k-\frac{1}{2}}
\prod_{1 \leq j < k} \frac{\ta_k^2 - \ta_j^2}{\left(k-\frac{1}{2}\right)^2 - \left(j-\frac{1}{2}\right)^2}\\
d_{\va(m)} &=
\prod_{k=1}^m d_\va(k).
\end{align*}
For those $k > m$, Lemma \ref{incremental_dimension_bound}
gives 
\[ \log d_\va(k) \leq (2k-1) \log
\alpha_k + O(n).\]

\begin{proof}[Proof of Lemma \ref{controlled_small_t}]
We may
write the LHS of (\ref{controlled growth small t}) as (recall $\sigma^2 = \frac{1 + O((\log n)^{-2})}{1+\omega^2}$)
\begin{align*} & m \Re \left( g_{\va(m)}\left(\omega' + i\frac{nt}{m}\right) -
g_{\vO(m)}\left(\omega'+i\frac{nt}{m}\right)\right)
\\&- \sum_{j = m+1}^n \Re \log \frac{\alpha_j^2 + (\omega^* +
it)^2}{1 + \omega^2} + O\left(\frac{n-m}{(\log n)^2}\right).
\end{align*}
Since $\rho_{\va(m)}$ is
either a small or a moderate
representation of $SO(2m+1)$, 
\[ m \Re \left( g_{\va(m)}(\omega' + it) -
g_{\vO(m)}(\omega'+it)\right) \leq  -  \frac{2 \sigma^2
\log d_{\va(m)} }{ m (\log m+C)}, \] and therefore the
proof is completed on observing that, for any fixed $c > 0$, for $n$ sufficiently large we have 
\[\sum_{j = m+1}^n \left[ \frac{2 \sigma^2
 \log d_{\va}(j)}{ n (\log n+C)}  -  \Re \log
\frac{\alpha_j^2 + (\omega' +
it)^2}{1 + \omega^2} \right]  < \frac{-c (n-m)}{(\log n)^2}\] in view of $t = o(1)$, the bound
 for $d_\va(j)$ above, and using $\alpha_j \geq 
10^6 \sqrt{1 
+
\omega^2}$ for all $m < j \leq n$.
\end{proof}

We prove the following estimate before proving Lemma \ref{C_integral_bound}.

\begin{lemma}
 Keep the definitions of $m$ and $\mathscr{C}_2$ from above
and let $m < j
 \leq n$.  We have
\[\inf_{z \in \mathscr{C}_2} \left|z^2 + \alpha_j^2\right| \geq 2 \left(\frac{1}{400}
\vee \omega^* \right)\alpha_j .\]
\end{lemma}
\begin{proof}
Observe that on any
line $\Re(z) = \alpha$, the minimum of $\left|z^2-(i \alpha_j)^2\right|$ is at least $2 \alpha
\alpha_j$.  It only remains to check that for
$\omega^* < \frac{1}{400}$,
the minimum of $\left|(\omega^* + it)^2 - (i\alpha_j)^2\right|$ for $t < \frac{m}{n}\frac{\sqrt{1 +
  {\omega^*}^2}}{3}$ exceeds $ 
\frac{\alpha_j}{200}$, but this is obvious geometrically.  
\end{proof}

We now bound the integral over the contour $\mathscr{C}_2$.

\begin{proof}[Proof of Lemma \ref{C_integral_bound}]
In view of the last lemma, the integral is  bounded by 
\begin{align*}& \exp\left(O\left(\frac{n-m}{(\log n)^2} + \log n\right)\right)\\& \times\prod_{j = m+1}^n
\frac{1 +
\omega^2}{2 \alpha_j \left(\frac{1}{400}\vee \omega^*\right)} 
 \int_{z \in \mathscr{C}_2} \exp\left(\Re
\left[
mg_{\va(m)}\left(\frac{n}{m}z\right) - mg_{\vO(m)}(\omega')\right]\right) d|z|  .\end{align*}

We claim that the latter integral is bounded by (note that the first factor covers the leading error above)
\[
 \exp\left(\frac{-(10^8 + o(1))m}{\log m}\right) \leq \exp\left(\frac{-c n}{\log n} -\frac{2 \sigma^2
\log d_{\va(m)} }{ m (\log m+C)}\right).
\]
If $\omega^*>
\frac{1}{400}$ then this
 follows from Lemma \ref{integral_tail} of the
previous section.  If $\omega^* \leq \frac{1}{400}$ Lemma
\ref{integral_tail} still bounds the vertical part of
the integral that is nearest the real axis, so it
remains to bound the horizontal part, and the vertical
part that extends to $\pm i \infty$.
By Lemma \ref{critical point} 
\[g_{\vO(m)}\left(\omega' + \frac{i\sqrt{1
+ {\omega'}^2}}{3}\right) - g_{\vO(m)}(\omega') \leq \frac{-1}{40}.\] 
Using $\theta
< \pi$ and $n-m < m$ we find that throughout the horizontal part of
$\mathscr{C}_2$, where  $|t|= \frac{\sqrt{1 + \omega^2}}{3}$, the
integrand is
bounded by \[\exp\left(-\left(\frac{1}{40} - \frac{\pi}{200}\right) m + o(m)\right) <
\exp\left(-\frac{m}{200}\right).\] Therefore the
remainder of the integral contributes
$\exp\left(-\frac{m}{200} +o(m)\right)$ by mimicking the proof of Lemma \ref{integral_tail}.

Using $\frac{1}{2(a \max b)} \leq \frac{1}{\sqrt{a^2 + b^2}}$, each term in the product is bounded by
$ \frac{20 \sqrt{1 +
\omega^2}}{\alpha_j} \leq \left(\frac{\sqrt{1 +
\omega^2}}{\alpha_j}\right)^{\frac{1}{10}}.$
Moreover, the bound for $\alpha_j > 1$,
\[
 \log d_\va(j) \leq (2j-1)\log \alpha_j + O(n),
\]
implies that
\[
-\frac{2 \sigma^2
\sum_{j=m+1}^n\log d_{\va}(j) }{ n (\log n+C)}\geq
\frac{1}{10}\sum_{j = m+1}^n \log \left(\frac{ \sqrt{1 +
\omega^2}}{\alpha_j
}\right)
\]
for all $n$ sufficiently large, which suffices to complete the dimension increment.
\end{proof}

\subsubsection{Giant representations}

Once the representation is giant, trivial considerations
suffice to bound the character ratio.  
\begin{proof}[Proof of Proposition \ref{character_ratio_bound_proposition}]
We take the
integral contour on the line $\Re(z) = \omega \vee 1$ so
as to avoid nearby poles of the integrand, and put in a
sup bound for all but the first factor from the product,
reserving the first factor to ensure convergence.  This
yields 
\begin{align*} |r_\va(\theta)| &\leq  O\left(\sqrt{n
g_0^{(2)} (\omega)}\right)
  \int_{\Re(z) = 1 \vee \omega} 
e^{n \Re\left(g_\va(z) - g_\vO(\omega)\right)} d|z|
\\ & \ll \frac{\sqrt{n} e^{\theta (1\vee \omega
-\omega)n}}{\sqrt{1 + \omega^2}}\int
\left|\frac{\left(\frac{1}{n}\right)^2 +
\omega^2}{\alpha_1^2 + ((1
    \vee \omega) 
+ it)^2} \right| dt \sup_{\Re(z) = 1 \vee 
\omega} \prod_{j =2}^n \left| \frac{
\left(\frac{j-\frac{1}{2}}{n}\right)^2 +
\omega^2}{\alpha_j^2 + z^2} \right|  .
\end{align*}
The integral is $O(\omega \wedge \omega^2)$.  
If $\alpha_j < 1 \vee \omega$
then the denominator of the $j$th term in the product is
minimized at $\Re(z) = 0$ with minimum value $\alpha_j^2
+ (1 \vee \omega)^2$.  Thus the $j$th term is bounded by
1. Otherwise, if $\alpha_j \geq 1 \vee
\omega$ the minimum value of the denominator is $2
\alpha_j (1 \vee \omega) > \alpha_j\sqrt{1 + \omega^2}$.
 Therefore, we obtain the bound
\[ 
 \left|r_\va(\theta)\right| \leq O\left(n^{\frac{1}{2}}  e^{((1 \vee \omega) -
\omega)\theta n}\right)\prod_{\substack{j: \alpha_j > 1
\vee \omega  \\ j > 1}} \frac{\sqrt{1 + \omega^2}}{  \alpha_j
}.
\]

Lemma \ref{incremental_dimension_bound} gives the dimension bound
\[
  \log d_\va \leq O(n^2) + 2n \sum_{k: \alpha_k \geq 1}
\log \alpha_k.
\]
We
deduce that for giant representations
\begin{align}\label{giant_comparison}   \log |r_\va(\theta)|
+ \frac{2\sigma^2 \log
d_\va}{n \log n } &\leq - \sum_{\substack{j: \alpha_j >
1 \vee \omega \\ j > 1}} \left(  \left(1 - \frac{2\sigma^2}{\log n} \right)  \left(\log
\alpha_j -\frac{1}{2}\log (1 +\omega^2)\right)\right)\\&\notag \qquad 
+\delta_{\omega <
1}(1-\omega)
\theta n  
+ O\left(\frac{\sigma^2 |\log \sigma| n}{ \log n}\right).
\end{align}

By virtue of being giant, for
\[
 k = n- \left\lfloor \frac{n}{2 \sqrt{1 + \omega^2}}\right\rfloor
\]
we have
\[
 a_k \geq \frac{2 \cdot 10^6 k}{\sqrt{1 + \omega^2}}, \qquad \Rightarrow \qquad \alpha_k \geq \frac{10^6}{\sqrt{1 + \omega^2}}.
\]

First consider $\omega < 1$.  In this case, the sum over $j \geq k$ contributes 
\[
 \lesssim - \frac{6 \ln 10}{2\sqrt{2}} n < - 4.88 n
\]
while the sum over $j < k$ contributes
\[
 \lesssim n \frac{\ln 2}{2} < .35 n.
\]
Since $.35 + \pi < 4.88$, (\ref{giant_comparison}) is asymptotically negative.

For $\omega>1$, the sum over $j \geq k$ contributes
\[
 \lesssim -\frac{ 3\ln 10}{\sqrt{1 +\omega^2}} n < \frac{-6.9 n}{\sqrt{1 + \omega^2}}
\]
while the sum over $j < k$ contributes
\[
 \lesssim n \ln \left(\sqrt{1 + \frac{1}{\omega^2}}\right) \leq  \frac{n}{2 \omega^2} < \frac{n}{\sqrt{1 + \omega^2}}.
\]
Thus, again, (\ref{giant_comparison}) is asymptotically negative.

\end{proof}

\subsection{The case of $\theta$ close to $\pi$}\label{close_to_pi}
Our treatment of moderate and large representations in the last two sections
completes the proof Theorem \ref{fixed_theta_theorem} for $\theta$ varying with
$n$ in the range $\frac{\log n}{\sqrt{n}} \leq \theta \leq \pi -
\frac{(\log
n)^2}{n}$. The case $\theta \approx \pi$ is difficult because the
location of the saddle point is harder to determine. We now
show that a modification of Rosenthal's argument in the 
case $\theta = \pi$ suffices to cover the range $\theta \geq \pi - \frac{(\log
n)^2}{n}$.

For this range of $\theta$, $\sin\frac{\theta}{2} = 1 - O(n^{-2 + \epsilon})$, and so
we seek the estimate
\[ \log |r_\va(\theta)| \leq - \frac{2 \log d_\va}{n(\log n + O(1))}.\]  The
section on small representations gives this estimate already for any $\va$
satisfying $ \sum_j a_j \leq \frac{n}{ \log n}$, so we may
assume that this does not hold. Applying  Lemma \ref{first_dimension_lower_bound} we deduce that 
\[
  \log d_\va \gg \frac{n}{\log n}.                                                                                               
\]

Introduce
\[ \overline{r}_\va(\pi) = \frac{(2n-1)!}{2^{2n-1}}
\sum_{j=1}^n
\left|\frac{\sin (\ta_j \pi)}{\ta_j \prod_{r \neq j} (\ta_r^2 - \ta_s^2)
}\right|.\]
 Note that
\[ r_\va(\theta) \leq \left(\sin \frac{\theta}{2}\right)^{-2n+1} \overline{r}_\va(\pi)
\leq \left(1 + O\left(\frac{(\log n)^2}{n}\right)\right) \overline{r}_\va(\pi).\]  Rosenthal's upper
bound in the case $\theta = \pi$ is proven by showing that 
$\log \overline{r}_\va(\pi) \leq \frac{-2 \log d_\va}{n(\log n + C)}.$
Therefore, bounding the remaining factor in $r_\va(\theta)$,
\[ \log |r_\va(\theta)| \leq O\left(\frac{(\log n)^2}{n}\right) - \frac{2 \log
d_\va}{n(\log n + C)} \leq -\frac{2 \log
d_\va}{n(\log n + C)} \left(1 + O\left(n^{-1 + \epsilon}\right)\right).\]  Thus the error may be
absorbed into the constant.

\section{Mixture of rotations  }\label{mixture_section}
The remainder of the paper concerns random walks wherein at
each step the angle of rotation is chosen from a fixed
distribution. If $\xi$ is this probability distribution on
$\T_0$, the  mixture walk
$P_\xi$ has Fourier
coefficients
\[ \xi(r_\va) = \int_{\T_0} r_\va(\theta) d\xi(\theta).\] 
Generically we expect that the mixing time in total variation is controlled
by the eigenvalue at the lowest dimensional non-trivial representation.  In
this case, this is the natural representation of dimension
$2n+1$, with eigenvalue
\[ \xi(r_{(\vO, 1)}) = \int_{\T_0} \left(1 - \frac{4\left(\sin
\frac{\theta}{2}\right)^2}{2n+1}\right) d\xi(\theta) = 1- \frac{4
\xi\left(\sigma^2\right)}{2n+1},\]
leading to a predicted total variation mixing time of
\begin{equation}\label{L^1 prediction}\frac{\log
d_{(\vO,1)}}{-\log \xi(r_{(\vO,1)})} \sim\frac{n\log
n}{2\xi(\sigma^2)}.\end{equation}

In the case of the mixture walk, the situation is
less clear because the quantity $\frac{\log d_\va}{-\log
|\xi(r_\va)|}$ is not necessarily maximized at the natural
representation.  Heuristically this is suggested by our
Proposition 
\ref{character_ratio_bound_proposition}, which proves the
bound 
\[ |\xi(r_\va)| \leq
\int_{\T_0} d_\va^{\frac{-2 \sigma^2}{n(\log n + C)}} d\xi(\theta).\]
This is only an upper bound, but for large representations $\rho_\va$
this bound 
suggests that $\xi(r_\va)$ is
largely controlled by the part of $\xi$ nearest 0 (the
issue is that the integration is not in the exponent). 

In the proof that follows we confirm the natural
representation prediction (\ref{L^1 prediction}) for the
mixing time in total variation, proving Theorem \ref{L^1 cut-off}, but as the
above discussion suggests, we do not follow the customary
path of bounding the $L^1$ norm with the $L^2$ norm,
instead using a truncation argument to bypass the larger
dimensional representations.  In the following section we
prove a cut-off for the $L^2$ norm at a point which
depends on the smallest point in the support of
measure $\xi$, thus confirming Theorem \ref{L^2 cut-off}.

\subsection{Random $\theta$  in total variation:
Proof of Theorem \ref{L^1
cut-off}}\label{L_1_section}
The lower bound follows from a standard application of
the second moment method (see \cite{DiaconisShahshahani})
applied to the function $\chi_{(\vO, 1)}$.   The necessary
estimates appear in Example \ref{char_ratio_example}:
\begin{align*}
d_{(\vO, 1)} = 2n+1, \qquad & \xi(r_{(\vO,1)}) = 1 - \frac{2
\xi(\sigma^2)}{n} + O\left(n^{-2}\right)
\\ d_{(\vO, 1,1)} = n(2n+1), \qquad & \xi(r_{(\vO, 1,1)}) =
1 - \frac{4 \xi(\sigma^2)}{n} + O\left(n^{-2}\right)\\
d_{(\vO, 2)} = n(2n+3), \qquad
&\xi(r_{(\vO, 2)}) = 1 - \frac{4\xi(\sigma^2)}{n} +
O\left(n^{-2}\right)
\end{align*}
together with the decomposition \[\chi_{(\vO,1)}^2 =
\chi_\vO + \chi_{(\vO, 1,1)} + \chi_{(\vO, 2)}.\]  We refer
the reader to Rosenthal's proof of the lower
bound in the case of deterministic $\theta$,
\cite{Rosenthal} Theorem 2.1, where the details are
 the same.

For the upper bound, recall that we set $\mu_\theta =
\delta_{\Id} \cdot P_\theta$ for the generating measure of
the fixed-$\theta$ walk.  Proposition
\ref{character_ratio_bound_proposition} guarantees
that there exists a $C>0$ such for all $n$ sufficiently
large, and for
all non-trivial representations $\va$, 
\[\log |\hat{\mu}_\theta(\chi_\va)| = \log
|r_\va(\theta)| \leq -
\frac{2\sigma^2(\theta)}{n(\log n + C)}\log d_\va.\]

Let $c > 0$ and let $t= \frac{n (\log n + 2C + 2c)}{2
\xi(\sigma)^2}$.
Conditioning on the choices of $\theta$ at each step of
the walk, and applying the triangle inequality, we have 
\[ \left\|\delta_{\Id}\cdot P_\xi^t   - \nu \right\|_{\TV}  \leq
\iiint\limits_{\vtheta \in
  \T_0^t} \|\mu_{\theta_1} \ast \cdots \ast
\mu_{\theta_t} - \nu\|_{\TV}
d\xi(\theta_1) \cdots d\xi(\theta_t).\] 
Since the total variation distance is bounded by 1, for any
measurable
set $E \subset \T_0^t$ we obtain the bound 
\[ \left\| \delta_{\Id} \cdot P_\xi^t - \nu \right\|_{\TV} \leq \xi^{\otimes t}(E) +
\iiint\limits_{\vtheta \in  E^c}
\left(\frac{1}{4}
  \sum_{\va \neq \vO} d_\va^2 \prod_{j=1}^t
  |\hat{\mu}_{\theta_j}(\chi_\va)|^2
\right)^{\frac{1}{2}}d\xi(\theta_1)\cdots
d\xi(\theta_t),\] by applying the Upper Bound Lemma on
$E^c$.  We now
define  
\[E = \left\{\vtheta \in (\T_0)^t: \sum_{j=1}^t
2\sigma^2(\theta_j)
  \leq n (\log n + C + c )\right\}.\] Standard tail estimates give that there exists a constant $K>0$ such that
\[\xi^{\otimes t}(E) \ll \exp\left(-\frac{(C+c)^2 n}{K\log
    n}\right).\]  Meanwhile, applying Proposition
\ref{sum_of_dimension_bound}, the square of the integrand over $E^c$ is
bounded by 
\begin{align*}&\sup_{\vtheta\in E^c}
  \frac{1}{4}\sum_{\va \neq \vO} \exp\left(2 \log
d_\va\left[1 -
      \frac{1}{n(\log n  + C)} \sum_{j = 1}^t
      2\sigma^2(\theta_j)\right]\right) \ll
\sum_{\va
    \neq \vO} d_\va^{\frac{-c}{\log n}} = O\left(e^{-\frac{c}{8}}\right), 
\end{align*}
completing the proof.

\subsection{Random $\theta$ walk in $L^2$: Theorem
\ref{L^2 cut-off}}\label{L_2_section}
Our proof of Theorem~\ref{L^2 cut-off} exhibits a competition between 
character ratio and dimension growth for small and moderate representations. 
The large representations are inconsequential. We 
first prove the lower bound, which is illustrative. Then we
discuss how to modify the argument from the fixed $\theta$
setting to obtain
the upper bound.

In this section, it is convenient to use as
reference points block representations of the form 
$ \va(s,t) = (\vO_{n-n^t},
(n^s)_{n^t})$, $0 < s, t < 1$, where $n^s,n^t$ are assumed to be integers. We write $\va(s) = \va(s,s)$.  In Lemma \ref{block_dimension} of Appendix \ref{dimension_section} we give the estimate
\begin{equation}\label{block_dimension_bound}
 \log d_{\va(s,t)} = (1 - s\vee t)
  n^{s+t} \left(\log n+O(1)\right).  
\end{equation}

The estimate regarding character ratios which we need for the lower bound in Theorem \ref{L^2 cut-off} is the following one.
\begin{proposition}\label{a_u_ratio_eval}
 Uniformly for $\theta$ in compact subsets of $(0,\pi)$ and for $\frac{1}{2} 
\leq u \leq \frac{3}{4}$ we have
\[  r_{\va(u)}(\theta) = \left(1 + O\left(n^{4u-3}\right)\right) \exp\left( -2 \left(\sin \frac{\theta}{2}\right)^2
n^{2u-1}\right).\]
\end{proposition}

\begin{proof}[Proof of lower bound in Theorem \ref{L^2 cut-off}]
We give only the proof for the sharper lower bound assuming a positive one-sided derivative of $\xi$ at $q$, the lower bound for general $\xi$, which requires only the leading order term, being easier.  For this bound it suffices to prove the
following two results separately:
\begin{enumerate}
 \item For $t = \frac{1}{2 \xi(\sigma^2)} n (\log n -c)$,
\[\left\|\delta_{\Id} \cdot P_\xi^t - \nu \right\|_2 \to
\infty, \qquad n \to \infty,\; c\to \infty\]
 \item For $t = \frac{1}{4 \sigma^2(q)} n ( \log n -3\log\log 
n- c )$, \[\left\|\delta_{\Id} \cdot P_\xi^t - \nu\right\|_2 \to \infty,
\qquad n \to \infty,\; c \to \infty.\]
\end{enumerate}
We prove both estimates by dropping all but one term
in the sum
\[
 \left\|\delta_{\Id} \cdot P_\xi^t - \nu\right\|_2^2 = \sum_{\va \neq \vO}
d_\va^2 \left|\xi(r_\va)\right|^{2t}.
\]

In the first case,  take the natural representation $\va
= (\vO,1)$, with dimension $d_{(\vO,1)} = 2n+1$
and character ratio $\xi(r_{(\vO,1)}) = 1 -\frac{2
\xi\left(\sigma^2\right)}{n} +O\left(\frac{1}{n^2}\right)$, which plainly suffices.  

Notice that the second case is contained in the first
unless $\sigma^2(q) < \frac{\xi\left(\sigma^2\right)}{2} \leq \frac{1}{2}$,
so we assume $\sigma^2(q) < \frac{1}{2}$ from now on.  In
this case we consider a representation of the form
$\va(u)$, where $u$ is slightly larger than $\frac{1}{2}$.
Introduce 
parameters $\Delta$ and 
$\beta$ characterized by
\[e^\beta = \frac{1}{2} \Delta \log \Delta = \frac{\log n}{2 \left(\sin
  \frac{q}{2}\right)^2}\] and 
set $u = \frac{1}{2} + \frac{\beta}{2 \log n}$.  
Note that our dimension evaluation in (\ref{block_dimension_bound}) gives
\[ 
 \left(1 + O\left(\frac{1}{\log n}\right)\right)\log d_{\va(u)} =\left(\frac{1}{2} 
- \frac{\log \log n}{2 \log n}\right) e^\beta n \log
n.
\]

Since $\xi$ has positive one-sided derivative at $q$, we have
\[\xi\left(\left[q,
q+ \Delta^{-1}\right]\right) \geq \frac{\delta}{\Delta}\] for some
$\delta > 0$.  Proposition \ref{a_u_ratio_eval} gives an
asymptotic evaluation of $r_{\va(u)}(\theta)$ for $\theta
\in [q, \pi - \epsilon]$, so in evaluating
the contribution of $\theta \in [\pi -
\epsilon, \pi]$ to $\xi(r_{\va(u)})$ we simply put in
the bound of Proposition
\ref{character_ratio_bound_proposition} for the
fixed $\theta$ walk, \[\log |r_{\va(u)}(\theta)| \lesssim
-\frac{\left(2 - \frac{\epsilon^2}{2}\right)\log
d_{\va(u)}}{ n(\log n +O(1))}.\]  Putting in the asymptotic
of Proposition \ref{a_u_ratio_eval} for $r_{\va(u)}(\theta)$
for $\theta$ on the small interval $\left[q, q+ \Delta^{-1}\right]$ and
using only that the character ratio is non-negative on the
remaining bulk, $\theta\in \left[q + \Delta^{-1}, \pi -
\epsilon\right]$, we deduce the lower bound
\begin{align*}\xi(r_{\va(u)}) &\geq \left(1 + O\left(n^{3u-2}\right)\right)
\int_{\theta
    \in \left[q, q 
+ \Delta^{-1}\right]} \exp\left(-2 \left(\sin \frac{\theta}{2}\right)^2 e^\beta\right)d\xi(\theta)\\ & 
\qquad\qquad  -
 \int_{|\theta - \pi| < \epsilon} \exp\left(-\left(1 +
O\left(\epsilon^2\right)\right)
e^\beta\right)d \xi(\theta)\\
& \geq \delta \Delta^{-1} \exp\left(-2 \left(\sin \frac{q}{2}\right)^2
e^\beta \right) \exp\left(-2\Delta^{-1} e^\beta
\right)-\exp\left(-\left(1 +
O\left(\epsilon^2\right)\right)e^\beta \right).
\end{align*}
Since $\sigma^2(q) < \frac{1}{2}$ this furnishes an
asymptotic of shape 
\[\xi(r_{\va(u)}) \gtrsim
\frac{\delta}{\Delta^2}\exp\left(-2 \left(\sin \frac{q}{2}\right)^2
e^\beta \right) \]
if $\epsilon$ is chosen sufficiently
small [use $\exp(2 \Delta^{-1} e^\beta) = \Delta \asymp
\frac{\log n}{\log \log n}$ while $e^\beta \asymp \log n$].
Therefore
\[ \log \xi(r_{\va(u)}) \geq -2 \left(\sin \frac{q}{2}\right)^2
e^\beta  - 2 \log\log n  -
O(1) = -\log n -2 \log \log n + O(1).\]
Recall that
\[
 \log d_{\va(u)} = \frac{1}{2 \left(\sin \frac{q}{2} \right)^2}\left(\frac{1}{2} - \frac{\log \log n}{\log n} + O\left(\frac{1}{\log n} \right) \right)n (\log n)^2
\]

We deduce that
\[ \frac{n(\log n - 3\log\log n)}{4 \left(\sin \frac{q}{2}\right)^2}
\log \xi(r_{\va(u)}) + \log
d_{\va(u)} \geq O\left(\log d_{\va(u)} (\log n)^{-1}\right). \]  The result follows since
the error can be absorbed in the constant $c$.
\end{proof}

The proof of Proposition \ref{a_u_ratio_eval} is by a direct saddle point evaluation.  We first prove a preliminary lemma.
\begin{lemma}
 Let $D$ be a bounded rectangle contained in the right half plane $\{z \in
\C: \Re(z) >0\}$.  Let $\frac{1}{2} \leq u < 1-\delta$ for some
$\delta > 0$,
and assume $n^u$ is an integer.  Uniformly for $z \in D$ we have for
all $m \geq 
1$ \[g_{\va(u)}^{(m)}(z) = g_\vO^{(m)}(z) - n^{2(u-1)} g_\vO^{(m+2)}(z) + O_{m,
D}\left(n^{4(u-1)}\right).\]  In particular, suppose $\theta \in (\epsilon, \pi -
\epsilon)$ for some $\epsilon > 0$.  Let $\omega' > 0$ solve the saddle point
equation $g_{\va(u)}'(\omega') = 0$ and let $\omega$ be the usual saddle point
$g_\vO'(\omega) = 0$.  We have
\[\omega' = \omega + n^{2(u-1)}
\frac{g_\vO^{(3)}(\omega)}{g_\vO^{(2)}(\omega)} + O_\epsilon\left(n^{4(u-1)}\right).\] 
Also, 
\[ g_{\va(u)}(\omega') = g_\vO(\omega) - n^{2(u-1)}g_\vO^{(2)}(\omega) +
O_\epsilon\left(n^{4(u-1)}\right), \qquad g_{\va(u)}^{(2)}(\omega') =
g_\vO^{(2)}(\omega) + 
O_\epsilon\left(n^{2(u-1)}\right).\]  For all $t > 0$ and all $n$ sufficiently
large we have \[\Re g_{\va(u)}\left(\omega' + it + \frac{i}{n}\right) \leq \Re
g_{\va(u)}(\omega' + 
it).\] 
Finally, there are constants $c_1, c_2 > 0$ such that if $t > c_1$ then 
\[
 \Re\left( g_{\va(u)}(\omega + 2it) - g_{\va(u)}(\omega + it)\right) < -c_2.
\]

\end{lemma}
\begin{proof}
Set $\phi(z) = i\left[ \psi(n
+ \frac{1}{2} -inz) - \psi(-n +\frac{1}{2} -inz)\right]$, and recall that 
$g_{\vO}'(z) = \theta + \phi(z)$.  We have 
\[g_{\va(u)}'(z) = g_{\vO}'(z) + \left(\phi\left(z + in^{u-1}\right) - 2 \phi(z) +
\phi\left(z-in^{u-1}\right)\right).
\]
 The term in parentheses is $n^{2(u-1)} g_{\vO}^{(3)}(z) + O\left(n^{4(u-1)}\right)$ by Taylor expansion.  The claims
for the other derivatives follow similarly.

The facts regarding $\omega'$ and $g_{\va(u)}^{(m)}(\omega')$ may be deduced by
standard calculus.  The monotonicity property for $g_{\va(u)}$ on $\Re(z)
= \omega'$ is proven in the same way as the related claim for $g_\vO$, in Lemma
\ref{critical point},
that is,
\begin{align*} &\Re\left( g_{\va(u)}\left(\omega + i\left(t + \frac{1}{n}\right)\right) - g_{\va(u)}(\omega + it) \right)\\&= \log
\left|\frac{\omega + i\left(t - 1 + \frac{1}{2(n-n^u)}\right)}{\omega +i\left(t +
1+\frac{1}{2(n-n^u)}\right)}\right|\left|\frac{\omega + i\left(t + 1 +\frac{1}{2n}\right)}{\omega +i\left(t -1+
\frac{1}{2n}\right)}\right|\left|\frac{\omega + i\left(t - 1+\frac{1}{2(n+n^u)}\right)}{\omega +i\left(t +1+
\frac{1}{2(n+n^u)}\right)}\right|\\
&= O\left(n^{-3+2u}\right) + \log\left|\frac{\omega + i\left(t - 1+\frac{1}{2n}\right)}{\omega +i\left(t +
1+\frac{1}{2n}\right)}\right|<0.\end{align*}

The last statement may be checked geometrically.
\end{proof}

Using the last lemma, we can now evaluate $\log \left|r_{\va(u)}(\theta)\right|$.

\begin{proof}[Proof of Proposition \ref{a_u_ratio_eval}]
The assumption on $\theta$ ensures that $g_{\vO}^{(2)}(\omega) = O(1)$. 
Standard application of the saddle point method gives 
\[r_{\va(u)}(\theta) =\left(1 +O\left(\frac{1}{n}\right)\right) \frac{(2n-1)!}{\left(2n \sin \frac{\theta}{2}\right)^{2n-1}}
\frac{e^{n
g_{\va(u)}(\omega')}}{\sqrt{2\pi n g_{\va(u)}^{(2)}(\omega')}}.\]  Comparing
this to the integral for the trivial
representation at its saddle point
$\omega$ yields
\begin{align*} r_{\va(u)}(\theta) &= \exp\left(n\left(g_{\va(u)}(\omega') -
g_{\vO}(\omega)\right)\right) \left(1  + O\left(n^{2(u-1)}\right)\right) \\&=
\exp\left(-n^{2u-1}g_{\vO}^{(2)}(\omega) + O\left(n^{4u-3}\right)\right) \\&= \exp\left(-2
\left(\sin \frac{\theta}{2}\right)^2 n^{2u-1}+O\left(n^{4u-3}\right)\right).
\end{align*}
Here we use $g_{\vO}^{(2)}(\omega) = 2 \left(\sin\frac{\theta}{2}\right)^2\left(1 + O\left(\frac{1}{n}\right)\right)$,  with an error absorbed in the $O\left(n^{4u-3}\right)$ term. 
\end{proof}

\subsubsection{Upper bound for $L^2$ mixture walk}
We now turn to the upper bound in Theorem \ref{L^2 cut-off}.
In this section we prove the following variant of the
character ratio bound in
Proposition~\ref{character_ratio_bound_proposition} for
mixtures of character ratios.
\begin{proposition}\label{mixture_char_ratio_bound} Let
$\xi$ be a probability measure supported in  $(0, \pi)$ and
let $q> 0$ and $q' < \pi$ be the smallest and largest points
in its support. As usual,
let $\sigma(\theta) = \sin\frac{\theta}{2}$.  Let $\rho_\va$ be
a non-trivial irreducible representation.  There exists a
constant $C > 0$ for which the following bound holds.  
\[
 \log \xi(|r_\va|) \leq -\min\left( \frac{2
\xi(\sigma^2)}{n(\log n + C)} , \frac{4
\sigma^2(q)}{n(\log n + 2 \log\log  n + C)}\right) \log d_\va.
\] 
\end{proposition}
\noindent From this Proposition, the deduction of the upper
bound in Theorem~\ref{L^2 cut-off} is the same as for
Theorem~\ref{fixed_theta_theorem}.

As in the proof of
Proposition~\ref{character_ratio_bound_proposition} for
fixed angle, the proof of Proposition
\ref{mixture_char_ratio_bound} splits according as the
representation is small, moderate or large.  When the
representation is small, in the range $\sum a_j \leq
\frac{n}{\sigma(q) \log n}$, one may apply Lemma
\ref{small_char_ratio_formula} directly to deduce that, uniformly in $\theta$,
\[
 r_\va(\theta) = 1 - \frac{E_1 \sigma^2(\theta)}{n^2}\left(1 +
O\left(\frac{1}{\log n}\right)\right),
\]
and that, in this range, $E_1 = O\left(\frac{n^2}{\log^2 n}\right)$.  Thus
for small representations,
\[ \xi(|r_\va|) = \xi(r_\va) = 1 -
\frac{E_1 \xi(\sigma^2)}{n^2}\left(1 + O\left(\frac{1}{\log n}\right)\right)\] and
\[
 \log \xi(|r_\va|) = -\frac{E_1 \xi(\sigma^2)}{n^2}\left(1 +
O\left(\frac{1}{\log n}\right)\right).
\]
The proof that
\[
 \log \xi(|r_\va|) \leq - \frac{2
\xi(\sigma^2)\log d_\va}{n(\log n + C)}
\]
now goes through as before.

Similarly, when the representation is large, in the regime
where $a_n \geq \frac{2\cdot 10^6}{\sigma(q)} n$,
arguments similar to our previous ones
reduce the problem to the case of small and moderate
dimensions. Recall that previously when the representation
is large, we either build the character ratio out of a ratio
on a smaller group by appending weights, or bound the
integral trivially.  We leave it to the reader to check that
in our bounds from Section \ref{large_section} on large
representations, the incremental contributions to the log of
the dimension are dominated by the contributions to the
character ratio by a factor of at least $\log n$, so that
large representations may be ignored.

We now turn to the main case of moderate representations,
where  there are some new ideas. One idea
is that the greatest
trade-off between character ratio and dimension growth
occurs for $a_n$ of size about $n^{\frac{1}{2}}$.  For
smaller $a_n$, the character ratio is sufficiently
near 1 that there is no loss in integrating it
directly as opposed to its logarithm.  For larger
$a_n$ the character ratio beats the dimension by
a larger amount.  Another idea, already illustrated in our
proof of the $L^2$ lower bound, is that the dimension is
most difficult to control when the indices $a_j$ are clumped. We handle this case in Lemma
\ref{dimension_clump_bound} below.

Before
starting out we make several simplifying reductions. 
As before, we bound the character ratio $r_\va(\theta)$ for
$\theta \in [q, q']$
by bounding the associated integral,\footnote{Throughout
this section we  work with integrals
truncated at $|t|\ll_\xi \frac{1}{\sqrt{\log n}}$,
thus neglecting the tail. The necessary argument to
control the error from the tail is the same as in Proof of
Proposition \ref{character_ratio_bound_proposition}.}
\[ 
 |r_\va(\theta)| \leq \frac{(2n-1)!}{\left(2n \sin
\frac{\theta}{2}\right)^{2n-1}} \frac{1}{2\pi}\int_{|t| \ll_\xi
\frac{1}{\sqrt{\log n}}} e^{n \Re(g_\va(\omega + it))}dt.
\]
This we write as
\begin{equation}\label{L_integral}
 \frac{(2n-1)!}{\left(2n \sin
\frac{\theta}{2}\right)^{2n-1}} \frac{1}{2\pi}\int_{|t| \ll_\xi
\frac{1}{\sqrt{\log n}}} e^{L_\va(\theta, t)} e^{n \Re(
g_\vO(\omega + it))}dt
\end{equation}
where we introduce
\[
 L_\va(\theta,t) = n \Re\left(g_\va\left(\omega(\theta) + it\right) -
g_\vO\left(\omega(\theta)
+ it\right)\right) = \log\left|\prod_{j = 1}^n \frac{(\omega(\theta) +
it)^2 +
\omega_j^2}{(\omega(\theta) + it)^2 + \alpha_j^2}\right|
\]  
and also $L_\va(\theta) = L_\va(\theta, 0)$.

We now record several lemmas regarding $L_\va(\theta, t)$.

\begin{lemma} \label{individual perturbation}
 In the range $\omega = \Theta(1)$, $\alpha_j = O(1)$
and $|t| = o(1)$,
\begin{align*}
 \left|\frac{\omega_j^2 + ( \omega + it)^2}{\alpha_j^2 + (
\omega + it)^2}\right| =
1- \frac{\alpha_j^2 - \omega_j^2}{\alpha_j^2 +
\omega^2}\left[1 +
O\left(t^2\right)\right].
\end{align*}
In particular, if $\rho_\va$ is a moderate representation
and $\theta \in [q, q']$, $|t| \ll \frac{1}{\log n}$ then
\[
 L_\va(\theta, t) = L_\va(\theta)\left(1 +
O\left(t^2\right)\right).
\]

\end{lemma}

\begin{proof}
Write
\[ 
 \frac{\omega_j^2 + ( \omega + it)^2}{\alpha_j^2 + (
\omega + it)^2} =  1- \frac{\alpha_j^2 -
\omega_j^2}{\alpha_j^2 + \omega^2} \left[1 +
\left(\frac{\alpha_j^2 + \omega^2}{\alpha_j^2 + (\omega +
it)^2} - 1\right)\right].
\]
Since $1- \frac{\alpha_j^2 - \omega_j^2}{\alpha_j^2 +
\omega^2} = \Omega(1)$, the proof of the first claim is
completed by noting
that $\left(\frac{\alpha_j^2 + \omega^2}{\alpha_j^2 +
(\omega +
it)^2} - 1\right)$ has imaginary part that is $O(|t|)$ and
real part that is $O(t^2)$.

For the second statement, note that $\rho_\va$ moderate
implies $\alpha_j = O(1)$, while $\theta \in [q,q']$
implies $\omega = \Theta(1)$.  The first statement then
applies, and shows that
\[
 \log \left|\frac{\omega_j^2 + ( \omega + it)^2}{\alpha_j^2
+ (
\omega + it)^2}\right| = \left(1 + O\left(t^2\right)\right)\log 
\frac{\omega^2 + \omega_j^2}{\omega^2 +
\alpha_j^2 }. 
\]
The second claim follows on summing over $j$.
\end{proof}

\begin{lemma}\label{L_a_proportional_lemma}
 Let $\rho_\va$ be a moderate representation and let
$\theta \in[q,q']$.  Then 
\begin{enumerate}
\item $L_\va(\theta) \asymp
L_\va(q)$;  
\item $\frac{1}{\log n} \ll -L_\va(q)$ for all
moderate $\va$.
\end{enumerate}
\end{lemma}

\begin{proof}
 Since $\omega = \Theta(1)$ and $\alpha_j = O(1)$ for
all $j$ the first statement follows immediately from the
definition,
\[
 L_\va(\theta) = \sum \log \left(1- \frac{\alpha_j^2
- \omega_j^2}{\omega^2 + \alpha_j^2}\right).
\]

For the second statement, recall that $\va$ moderate
guarantees $\sum_j a_j \gg \frac{n}{\log n}$ and $a_n \ll
n$. We bound
\[
 -L_\va(q) = -\sum_{j} \log \left(1 -
\frac{\alpha_j^2 - \omega_j^2}{\omega^2 +
\alpha_j^2}\right)\gg \sum_{j} (\alpha_j - \omega_j)
\gg \frac{1}{\log n}
\]
since $1 \ll \omega_j \leq \alpha_j \ll 1.$
\end{proof}

\begin{lemma}\label{ratio_at_saddle_lemma}
 There is a constant $c > 0$ such that for all moderate
representations $\rho_\va$ we have \[\log \xi(|r_\va|) \leq
c L_\va(q).\]  Furthermore, there is a $c'(\xi) > 0$ such
that if $-L_\va(q) < c' n$ then 
\[
 \log \xi(|r_\va|) \leq \left(1 + O\left(\frac{\log n}{n}\right)\right)\log
\xi\left(e^{L_\va(\theta)}\right). 
\]
\end{lemma}

\begin{proof}
 By modeling the proof of Proposition
\ref{character_ratio_bound_proposition} for moderate
representations (see e.g. (\ref{two_part_bound})) we have
\[
 \xi(|r_\va|) \leq \left(1 + O\left(\frac{1}{n}\right)\right) \sup_{\theta \in [q,q']}
\sup_{|t| \ll_\xi \frac{1}{\sqrt{\log n}}}
e^{-L_\va(\theta, t)}.
\]
The first statement follows
since $L_\va(\theta, t) \sim L_\va(\theta) \asymp L_\va(q)$
[the error term of size $O(\frac{1}{n})$ is negligible].

To prove the second statement, use (\ref{L_integral}) to
write
\begin{align*}
 &|r_\va(\theta)| \leq e^{L_\va(\theta)} \frac{(2n-1)! e^{n
g_\vO(\omega)}}{\left(2n \sin\frac{\theta}{2}\right)^{2n-1}}
\\& \qquad \qquad \times\frac{1}{2\pi}\int_{|t| \ll_\xi
\frac{1}{\sqrt{\log n}}} \exp\left(-\frac{n
g_\vO^{(2)}(\omega) t^2}{2} + O\left(t^2 L_\va(\theta)\right) + O\left(n
t^4\right)\right) dt.
\end{align*}
For $L_\va(\theta) < c'' n$ for a sufficiently small
constant $c''$, the entire expression evaluates to (see
(\ref{absolute_integral}))
\[ e^{L_\va(\theta)} \left(1 + O\left(\frac{1}{n}\right) + O\left(\frac{L_\va(\theta)}{n}\right)\right).\]
Thus \[\xi(|r_\va|) \leq \xi\left(e^{L_\va}\right) \left(1 + O\left(\frac{1}{n}\right) +
O\left(\frac{L_\va(q)}{n}\right)\right).\]  The claim follows since
$\log \xi\left(e^{L_\va}\right) \asymp L_\va(q) $ and $-L_\va(q) \gg
\frac{1}{\log n}$.
\end{proof}

The above lemmas effectively reduce our problem to that of
comparing $L_\va(\theta)$ to the dimension $ \log d_\va$. 
Recall that earlier we introduced the dimension increment
\[
 d_\va(k) = \frac{\ta_k}{2k-1} \prod_{j < k} \frac{\ta_k^2
- \ta_j^2}{(k-\frac{1}{2})^2 - (j-\frac{1}{2})^2}
\]
and that in Lemma \ref{incremental_dimension_bound} we
proved the bound
\begin{equation}\label{dimension_increment_bound}
 \log d_\va(k)  \leq a_k \log \left(1 +
\frac{2k-1}{a_k}\right) + (2k-1) \log \left(1 +
\frac{a_k}{2k-1}\right) + O(1).
\end{equation}
Similarly, now we introduce the character ratio increment
(recall $\alpha_j = \frac{a_j + j-\frac{1}{2}}{n}$)
\[
 L_\va(k, \theta) = \log \left(1 - \frac{\alpha_j^2 -
\omega_j^2}{\omega(\theta)^2 + \alpha_j^2}\right).
\]
We now prove several lemmas about these increments.

\begin{lemma}\label{ratio_increment_bound}
 Let $\rho_\va$ be a moderate representation.  For all $j
\gg n$ and $\theta$ in the support of $\xi$ we have \[-L_\va(j, \theta) \gg_\xi \frac{a_j}{n}.\]
 If $j > n - \frac{Cn}{\log n}$ for some constant $C$ and $a_j < 
\frac{n}{\log n}$ then
\begin{equation}\label{L_increment_asymp} -L_\va(j, \theta)
= \frac{2\sigma^2(\theta) a_j}{n} \left(1 + O\left(\frac{1}{\log
n}\right)\right).\end{equation}
\end{lemma}

\begin{proof}
 For both claims, write $\alpha_j^2 - \omega_j^2 =
\frac{a_j}{n}\cdot \frac{a_j + 2j-1}{n}$.  The first claim
is  straightforward.  For the second, observe that
$\frac{a_j + 2j-1}{n} = 2 + O\left(\frac{1}{\log n}\right)$ and
$\frac{1}{\omega^2 + \alpha_j^2} = \sigma^2 \left(1 + O\left(\frac{1}{\log
n}\right)\right)$.
\end{proof}

The next lemma allows us to prove the type of estimate that
we want for Proposition \ref{mixture_char_ratio_bound} for
a collection of increments whose sum is not too large.
\begin{lemma}\label{small_bits_integral}
 Let $C$ be a constant and let $G \subset \zed \cap \left[n-\frac{Cn}{\log n}, n\right]$ be a subset of indices such that $\sum_{j \in G}
a_j < \frac{n}{\log n}$.  Then for a sufficiently large
constant $C'$,
\[
\log \int_{q}^{q'} \exp\left(\sum_{j \in G}
L_\va(j, \theta)\right)d\xi (\theta) \leq -\frac{2
\xi(\sigma^2) \sum_{j \in G} \log d_\va(j)}{n
(\log n + C')}.
\]
\end{lemma}

\begin{proof}
 Since $a_j < \frac{n}{\log n}$ for all $j \in G$, we may
substitute the asymptotic (\ref{L_increment_asymp}) into
the LHS and integrate.  Since the argument of the
exponential is $O\left(\frac{1}{\log n}\right)$, the LHS becomes
\[
\left(1 + O\left(\frac{1}{\log n}\right)\right) \frac{-2 \xi(\sigma^2)}{n} \sum_{j \in G}
a_j.
\]
The result now follows on noting that, in this range, $\log
d_\va(j) \leq a_j(\log n + O(1)).$
\end{proof}

The following dimension lemma is proved as Lemma \ref{dimension_clump_bound} of Appendix \ref{dimension_section}.
\begin{lemma*}
 Let $B$ be the collection of indices
\[
 B = \left\{j \in \left[n- \frac{Cn}{\log n}, n\right]: a_j \in \left(\left(1 - \delta\right)S,
S\right]\right\}
\]
where $\delta = \frac{1}{\log n}$, $C$ is a constant,  and
$1 \leq S \leq \frac{n^{\frac{1}{2}}}{1 + \log n}$.  Suppose that
$\frac{n^{\frac{1}{2}}}{\log n}\leq |B|  \leq \frac{Cn}{\log n}$.  Then
\[
 \log d_\va(B) \stackrel{def}{=} \sum_{j \in B} \log
d_\va(j) \leq \left(1 - \frac{\log |B|}{\log n}\right)
|B|S(\log n + O(1)).
\]
\end{lemma*}

In particular, for each $\theta \in [q,q']$ we have
\begin{equation}\label{char_ratio_bound_for_clump}
 \sum_{j \in B} L_\va(j,\theta) \leq \frac{-4 \sigma^2(q)
\log d_\va(B)}{n(\log n + 2 \log\log  n + O(1))}.
\end{equation}
Indeed, for each $j \in B$, Lemma
\ref{ratio_increment_bound} gives 
\[
 -L_\va(j,\theta) = \frac{2 \sigma^2(\theta) S}{n} \left(1 +
O\left(\frac{1}{\log n}\right)\right)
\]
so that
\[
 \sum_{j \in B} L_\va(j,\theta) \leq -\frac{2 \sigma^2(q)
|B|S}{n} \left(1 + O\left(\frac{1}{\log n}\right)\right),
\]
while 
\[
 d_\va(B) \leq \left(\frac{1}{2} + \frac{\log\log  n}{\log
n}\right) |B| S (\log n + O(1)).
\]

\begin{lemma}\label{char_ratio_shift_lemma}
 Let $1 \leq k \leq n$ and let  $\rho_\va$ be a moderate representation with $a_j =
0$ for $j < k$. Let $\va'$ denote the
string such that $a'(i) = a(i)$ for $i < k$ and $a'(i) =
a(i) + 1$ for $i \geq k$. For some constants $C, C' > 0$
\[
 \log d_{\va'} \leq \log d_\va + Cn
\]
while
\[ 
 L_{\va'} \leq L_\va - C'\left( \frac{n-k}{n}\right)
\]
\end{lemma}

\begin{proof}
 The dimension claim follows from Lemmas
\ref{shift_ratio_bound} and \ref{s_shift_bound}.  The
character ratio claim follows on checking that for each $i
\geq k$ we have $\log \left(1 - \frac{\alpha_i^2 -
\omega_i^2}{\omega^2 + \alpha_i^2}\right)$ decreases by
$\Omega(\frac{1}{n})$ after the shift.
\end{proof}

\begin{proof}[Proof of Theorem~\ref{L^2 cut-off} upper
bound, moderate representations]
We begin with a few observations.  We assume that
$\rho_\va$ is moderate, with $a_n \leq \frac{2 \cdot
10^6}{\sigma(q)} n$.  Therefore, by Lemma
\ref{moderate_dimension_bounds} of Section
\ref{moderate_section}, $\log d_\va \ll n^2$.  Therefore,
we may restrict attention to character ratios for which
$-\log \xi(|r_\va|) \ll \frac{n}{\log n}$.  By Lemmas
\ref{L_a_proportional_lemma} and \ref{ratio_at_saddle_lemma}
of this section, this means that we may assume $-L_\va(q)
\ll \frac{n}{\log n}$, so that, to within negligible error,
we may replace
\[
 \xi(|r_\va|) \qquad \leftrightarrow\qquad
\xi\left(e^{L_\va(\theta)}\right).
\]

Next observe that for a sufficiently large constant $c$, we
may assume $a_i = 0$ for all $i < n - \frac{cn}{\log n}$. 
This is because any string with non-zero entries in this
region may be obtained from one with all zeros by making
shifts of the type described in Lemma
\ref{char_ratio_shift_lemma}, and if $c$ is sufficiently
large, these shifts improve the bound in Proposition
\ref{mixture_char_ratio_bound}.

We next dispatch with any indices for which $a_j >
\frac{n^{\frac{1}{2}}}{\log n}$.  Recall that $a_j \ll n$ so that
(\ref{dimension_increment_bound}) gives \[\log d_\va(j) \leq
a_j(\log n - \log a_j + O(1)).\]  If $a_j \geq
\frac{n}{(\log n)^2}$ then the first part of Lemma
\ref{ratio_increment_bound} gives $L_\va(j, \theta) \gg
\frac{a_j}{n}$. Thus \[- \frac{n \log n}{4 \sigma^2(q)}
L_\va(j, \theta)\gg \frac{\log n}{\log\log  n} \log d_\va(j).\]
 If
$a_j \leq \frac{n}{(\log n)^2}$ then the second part of
Lemma \ref{ratio_increment_bound} guarantees that, for a
sufficiently large constant $C$,
\[
  \frac{n (\log n + 2 \log\log  n + C)}{4 \sigma^2(q)}
L_\va(j,\theta) + \log d_\va(j) \leq 0,
\]
but now the factor of $\log\log  n$ is needed.

Now we treat the indices for which $a_j < \frac{n^{\frac{1}{2}}}{\log n}$
by splitting them into bins.  Let $\delta = \frac{1}{\log n}$ and
let \[I_k = \left[\frac{n^{\frac{1}{2}}}{\log n} (1-\delta)^k,
\frac{n^{\frac{1}{2}}}{\log n}(1-\delta)^{k-1}\right], \qquad 1 \leq
k \leq K= \frac{\log \frac{n^{\frac{1}{2}}}{\log
n}}{-\log(1-\delta)} \sim \frac{1}{2}(\log n)^2.\]  Say an
interval $I_k$ is `big' if there are at least $\frac{n^{\frac{1}{2}}}{\log
n}$ indices $j$ for which $a_j \in I_k$.  By (\ref{char_ratio_bound_for_clump}), if $I_k$ is big then we have, for a
sufficiently large $C $, for all $\theta \in [q,q']$,
\[
 \sum_{j: a_j \in I_k} L_\va(j, \theta) \leq -\frac{4
\sigma^2(q)}{n(\log n + 2 \log\log  n + C)} \sum_{j : a_j \in
I_k} \log d_\va(j).
\]

Let $\mathscr{B}$ be the collection of all indices $j$
belonging to a big interval, together with all $j$ for
which $a_j > \frac{n^{\frac{1}{2}}}{\log n}$.  Let $m \geq n - \frac{cn}{\log n}$
be the least index for which $a_m \neq 0$, and set
$\mathscr{S} = [m,n] \setminus \mathscr{B}$.  
Observe that \[L_\va(\theta) =\sum_{j \in
\mathscr{B}}L_\va(j,\theta) + \sum_{j \in
\mathscr{S}}L_\va(j,\theta) = L_{\va,
\mathscr{B}}(\theta) + L_{\va, \mathscr{S}}(\theta),\] and
\[\log d_\va = \sum_{j \in \mathscr{B}}d_\va(j) + \sum_{j
\in \mathscr{S}}d_\va(j) = \log d_{\va, \mathscr{B}} + \log
d_{\va, \mathscr{S}}.\]

Since
each $j \in \mathscr{S}$ has $a_j$ in an interval that
contains at most $\frac{n^{\frac{1}{2}}}{\log n}$ other indices, we may
bound
\[
 \sum_{j \in \mathscr{S}} a_j \leq
\frac{n^{\frac{1}{2}}}{\log n}\sum_{k = 0}^{K-1}(1-\delta)^k
\frac{n^{\frac{1}{2}}}{\log n} \leq \frac{n}{\log n},
\]
and, therefore, by Lemma \ref{small_bits_integral}, 
\[
\log \int_q^{q'} e^{L_{\va, \mathscr{S}}(\theta)}
d\xi(\theta) \leq -\frac{2 \xi(\sigma^2) \log d_{\va,
\mathscr{S}}}{n(\log n + C)}.
\]
Now we bound 
\begin{align*}
 \log \xi(|r_\va|) &\leq \log \int_q^{q'} e^{L_\va(\theta)}
d\xi(\theta) \\&\leq  \sup_{\theta \in [q, q']} L_{\va,
\mathscr{B}}(\theta) + \log \int_q^{q'} e^{L_{\va,
\mathscr{S}}(\theta)} d\xi(\theta) \\& \leq
-\frac{4\sigma^2(q) \log d_{\va, \mathscr{B}}}{n(\log n + 2
\log\log  n + C)} -\frac{2 \xi(\sigma^2)\log d_{\va,
\mathscr{S}}}{n(\log n + C)},
\end{align*}
completing the proof.
  \end{proof}

\section{$L^\infty$ mixing time} \label{L^infty_section}
In this section we prove Corollary~\ref{L^infty cut-off}.
First we define, for two Borel probability measures on a
compact metric space $(\Omega, \rho)$, the
\emph{$L^\infty$ distance of $\mu$ and $\nu$ w.r.t.
$\nu$}, by 
\begin{align*}
\left\| \mu - \nu\right\|_\infty := \left\|\mu - \nu\right\|_{\infty,\nu} := \sup_{f: \nu(f)
  \le 1} \left|\mu(f) - \nu(f)\right|. 
\end{align*}
If $\mu$ is not absolutely continuous with respect to
$\nu$ then the $L^\infty$ norm is $\infty$, since if $U
=
\left\{x \in \Omega: \frac{d\mu}{d\nu}(x) = \infty\right\}$ then for
any Borel
function $f$ supported on $U$, $\nu(f+1) = 1$. 
Conversely, if $\mu$ has a density with respect to $\nu$,
then we have
the alternative definition: 
\begin{align*}
\left\| \mu -\nu\right\|_\infty = \esssup_{x
\in \Omega} \left|\frac{d\mu}{d\nu}(x) -1\right|.
\end{align*}
The equivalence of the two definitions is trivial in the discrete
case. For compact metric space, use the fact that
$\frac{d\mu}{d\nu}(x) = \lim_{r \to 0^+}
\frac{\mu(B_r(x))}{\nu(B_r(x))}$, where $B_r(x)$ is metric ball of
radius $r$ around $x$. 

Now specializing to the running distributions of a Markov chain, the
first definition further implies that for any starting state $x$, the
$L^\infty$ distance $d_\infty(t) := \left\| P_x^t - \nu\right\|_\infty$ is
monotone non-increasing, since  
\begin{align*}
\left\|P_x^{t+1} - \nu\right\|_\infty &= \left\|P_x^t P - \nu P \right\|_\infty = \sup_{f:
  \nu(f) \le 1}\left| P_x^t(P(f)) - \nu(P(f))\right| \\ 
 &\le \sup_{g: \nu(g) \le 1} \left|P_x^t(g) - \nu(g)\right| = \left\|P_x^t - \nu\right\|_\infty,
\end{align*}
using $\nu(f) \le 1$ implies $\nu(P(f)) \le 1$, by Markov property.

  Now recall the Fourier inversion formula for compact Lie groups: for
  any $f \in L^2(SO(2n+1))$, 
\begin{align*}
f(x) = \sum_{\rho \in \widehat{SO(2n+1)}} d_\rho \tr \left[\hat{f}(\rho) \rho(x)^*\right].
\end{align*}
Let $f_t = \frac{d\mu_t}{d\nu} -1 $, where $\mu_t = \delta_{\Id} \cdot 
P^t$. Since we are interested in the situation where
$\esssup_{x \in
  \Omega}\frac{d\mu_t}{d\nu}(x) < \infty$, certainly we may assume
$f_t \in L^2(SO(2n+1))$. Since $\mu_t = \mu^{\ast t}$ is a convolution
product and invariant under conjugation, $\hat{f_t}(\rho_\va) =
\xi(r_\va(\theta))^t I_{d_\va}$, that is, a constant times the
identity matrix. In particular, $\hat{f_t}(\rho_{\vO}) = 1$. Therefore  
\begin{align*}
f_t(x) = \sum_{\va\neq \vO} d_\va^2 r_\va^t \chi_\va(x);
\qquad \chi_\va(x) =
\frac{1}{d_\va} \tr \rho_\va(x).  
\end{align*}
Thus when $t = 2 s$ is even, the maximum of $\left|f_t(x)\right|$ occurs at $x
\in \{\pm \Id\}$ ($r_\va:= \frac{1}{d_\va}\tr\hat{\mu}(\rho)$ can be
either positive or negative). But at $x = \Id$ say, we have $f_t(\Id)
= \sum^*_\va d_\va^2 r_\va^t$, which coincides with $\left\|P_\Id^s -
\nu_n\right\|_2$. Thus $\left\|P_\Id^t - \nu_n\right\|_\infty = \left\|P_\Id^s - \nu_n\right\|_2$
and the Corollary follows from monotonicity of the
$L^\infty$ distance. 

\section{Open problems}
We have seen the power of the contour representation formula
\eqref{contour_integral_formula} in proving sharp estimates for the
size of the character ratio, which is instrumental in studying the
convergence rate of the uniform plane Kac walk under various
norms. Several issues still remain: 
\begin{itemize}
\item  What is the true cut-off window of $L^2$ and $L^\infty$
  convergence? We show that it is of order $O(n
\log\log  n)$ for measures of a specific type.  For a sum of point masses, we may expect that $O(n)$ is the true order.  Does a cut-off exist regardless of the measure?
\item  Theorem~\ref{L^2 cut-off} artificially assumes that the generating measure $\mu$ be supported away from $\pi \in \T_0$. Probably this is an artifact of our proof and can be removed with more effort. 
\item  For a fixed $n$ it would be nice to determine the
first step at which the $L^2$ distance from uniform
becomes finite.  Carefully tracing our treatment of large
representations shows that this happens in time $O(n)$,
while a time $\geq n$ is necessary, because prior to this
point the measure is supported on a lower dimensional
submanifold.
\item In Appendix \ref{general contour SO(N)} we prove a generalization of the
the contour formula
  \eqref{contour_integral_formula} to the case where the generating
  measure of the random walk is supported on bigger conjugacy classes
  of $SO(2n+1)$, such as $SO(2k) \subset SO(2n+1)$. It remains to be seen whether
multidimensional analogues of our arguments yield the
corresponding mixing time analysis for the related random
walks.  A reasonable lower bound and candidate mixing
time is available in this case from the second moment
method applied to the natural representation.

\item In \cite{GYChen} it is proved that under very general
conditions on a
finite state space Markov chains, $\ell^p$-cut-off implies
$\ell^q$-cut-off for $q > p$. Does the same hold for our continuous
state space walk?
\end{itemize}

\appendix

\section{Bounds for dimensions}\label{dimension_section}
This section collects together the basic estimates for dimensions of
representations that we  need throughout the mixing time upper
bound argument.  We conclude by 
proving Proposition \ref{sum_of_dimension_bound}.

Since $d_\vO = 1$, the dimension equation (\ref{dimension formula}) may be
written as\footnote{Recall $\ta_j = a_j + j-\frac{1}{2}$.}
\[  d_\va = \prod_{k = 1}^n \frac{\ta_k}{k-\frac{1}{2}} \prod_{1 \leq j < k \leq
n}\frac{\ta_k^2 - \ta_j^2}{\left(k-\frac{1}{2}\right)^2 - \left(j-\frac{1}{2}\right)^2}.\] We may treat this product
incrementally as
\[ d_\va = \prod_{k = 1}^n d_\va(k), \qquad d_\va(k) = \frac{\ta_k}{k-\frac{1}{2}}
\prod_{1 \leq j < k} \frac{\ta_k^2 - \ta_j^2}{\left(k-\frac{1}{2}\right)^2 - \left(j-\frac{1}{2}\right)^2}.\]
If we replace $\ta_j$ with $j-\frac{1}{2} \leq \ta_j$ in the
product above, then we find
\begin{equation}\label{dimension_trivial_incremental} d_\va(k) \leq \frac{a_k + k - \frac{1}{2}}{k-\frac{1}{2}}
\frac{(a_k + 2k - 2)!}{a_k! (2k-2)!}.\end{equation}  In
particular, this gives the following weak estimate for the
dimension.
\begin{lemma}\label{incremental_dimension_bound}
 Recall the notation $\alpha_k = \frac{\ta_k}{n}$.  We have, for
each $k$,
\[ \log d_\va(k) \leq a_k \log \left(1 +
\frac{2k-1}{a_k}\right) + (2k-1) \log\left(1+\frac{a_k
}{2k-1}\right) + O(1).\]  In particular, 
\[  \log d_\va \leq O(n^2) + 2n \sum_{k: \alpha_k \geq 1}
\log \alpha_k.\]
\end{lemma}

\begin{proof}
 The first line follows from Stirling's approximation, and the second follows
on summing.
\end{proof}

To get more refined estimates, we  frequently
make index-shift arguments, and so it is most convenient to 
 attach to representation $\rho_\va$ the shift-index
$\vs$ given by
\[\vs = (s_1, ..., s_n), \qquad s_1 = a_1, \quad s_i = a_i - a_{i-1},
\qquad i = 
2, ..., n.\] Note that as $\va$ runs over non-zero
weakly increasing strings of length $n$, $\vs$ runs over all non-zero
strings in 
$\N^n$, and $\va$ is recovered from $\vs$ as $a_j = \sum_{i \leq j}s_i$.  We
abuse notation by writing $d_\vs = d_\va$ for $\vs$  associated to $\va$.

Treating the product $d_\vs$ now as a whole, we may split it as
$d_\vs = d_\vs^- d_\vs^+ d_\vs^0,$ 
where
\begin{align*}
 d_\vs^- &= \prod_{1 \leq j < k \leq n} \frac{\left[\sum_{j < i \leq k}
s_i\right] + k-j}{k-j}\\
d_\vs^+ &= \prod_{1 \leq j < k \leq n} \frac{\left[ \sum_{i \leq j} s_i +
\sum_{i \leq k} s_i\right] + j+k-1}{j+k-1}\\
d_\vs^0 &= \prod_k \frac{\left[\sum_{i \leq k} s_i\right] + k-\frac{1}{2}}{k-\frac{1}{2}}
\end{align*}
Note that $d_\vs^* \geq 1$.

Let $\ve_i$ denote the $i$th standard basis vector in $\N^n$.  Our
first set of bounds concern dimensions of shifts supported at the
standard basis vectors. 
\begin{lemma}\label{shift_ratio_bound}
 For any $\vs_0 \in \N^{n-i}$  we have
\[ \frac{d_{(\vO_{i},  \vs_0)+ \ve_{i+1}}}{d_{(\vO_i,
\vs_0)}} \leq d_{\ve_{i+1}}.\] 
\end{lemma}
\begin{proof}
 This is immediate from comparing separately each factor $d_\vs^-$,
 $d_\vs^+$ and $d_\vs^0$. 
\end{proof}

We can easily bound $d_{ \ve_i}$.
\begin{lemma}\label{s_shift_bound}
Let $m = \min(i-1, n-i + 1)$.  We have the bound
\[
 \log d_{ \ve_i} \leq m\left[1 + \log 2 + \log
\frac{n}{m} + \log
\frac{n}{i-\frac{1}{2}}\right] + 2(n-i)[1 + \log 2] + \log \frac{n+\frac{1}{2}}{i-\frac{1}{2}} .
\]
In particular, $\log d_{\ve_i} = O(n)$.
\end{lemma}
\begin{proof}
 Splitting $d_{ \ve_i} = d_{\ve_i}^{-}d_{ \ve_i}^{+} d_{
   \ve_i}^{0}$ as above, we bound each in turn. 
\begin{align*}d_{\ve_i}^{-} & = \prod_{1 \leq j < i}\prod_{i \leq k
    \leq n} \frac{1+ k-j}{k-j} 
\leq \exp\left( \sum_{1 \leq j < i} \sum_{i \leq k \leq n}
  \frac{1}{k-j}\right), 
\end{align*}
so dividing according to $k-j \le m$ and $k-j > m$ below, we have
\[ \log d_{ \ve_i}^{-} \leq \sum_{1 \leq j < i}\sum_{i \leq
  k \leq n} 
\frac{1}{k-j} \leq m + m\left[\frac{1}{m+1} + ... + \frac{1}{n-1}\right] \leq
m\left[1 + \log \frac{n}{m}\right].
\]
Similarly,
\begin{align*}
 d_{ \ve_i}^{+} &= \prod_{1 \leq j < i} \prod_{i \leq k \leq n}
 \frac{ j+k}{j+k-1} \cdot \prod_{i \leq j < k \leq n} \frac{
   j+k+1}{j+k-1} 
\\& \leq \exp\left(  \sum_{1 \leq j < i} \sum_{i \leq k \leq n}
    \frac{1}{j+k-1} + 2 \sum_{i \leq j < k\leq n}
    \frac{1}{j+k-1}\right) 
\end{align*}
so that
\begin{align*}
 \log d_{\ve_i}^{+} &\leq \sum_{1 \leq j < i}\sum_{i \leq
   k \leq n}\frac{1}{j+k-1} + 2 \sum_{i \leq j < k\leq n}
 \frac{1}{j+k-1} 
\\& \leq m\left[\frac{1}{i} + \frac{1}{i+1} + ... +
\frac{1}{2n-1}\right] + 2\left[\frac{1}{2i} + \frac{2}{2i+1} + ... +
\frac{n-i}{n+i-1}\right]
\\& \qquad\qquad\qquad\qquad  + 2(n-i)\left[\frac{1}{n+i} +
\frac{1}{n+i+1} + ...\frac{1}{2n-1}\right] \\&\leq m \log \frac{2n}{i-\frac{1}{2}} +
2(n-i)\left[ 1 + \log 2\right] .
\end{align*}
Finally,
\[d_{\ve_i}^{0} = \prod_{k \geq i} \frac{ k+\frac{1}{2}}{k-\frac{1}{2}}=\frac{n+\frac{1}{2}}{i-\frac{1}{2}}.
\]

\end{proof}
As a consequence of this lemma we get a crude bound for the dimension of
any representation.
\begin{corollary}\label{crude_dimension_bound}
For all sufficiently large $n$, each representation $\rho_\va$ has dimension
bounded by
\[d_\va \leq e^{10 n a_n}.\]
If $a_i = 0$ for all $i \leq \frac{n}{2}$ then
\[ d_\va \leq \exp\left( |\va|\left( \log n + 7 \right)\right).\]
\end{corollary}
\begin{proof}
 Note that in shift notation, $a_n = |\vs| = \sum_j s_j$.  By Lemma
\ref{shift_ratio_bound},
\[
 d_\vs \leq \prod_{i=1}^n d_{s_i \ve_i}.
\]
The maximum over $i$ of the upper bound for $s^{-1} \log d_{s \ve_i}$ in Lemma
\ref{s_shift_bound} is less than $10 n$ for all $n$ sufficiently large.
Combining these observations, the first claim of the corollary follows.

To prove the second, observe that $|\va| = \sum_{i} (n-i+1)s_i$.  For all $i >
\frac{n}{2}$ Lemma \ref{s_shift_bound} gives $\log d_{s_i \ve_i} \leq
(n-i+1)s_i(\log n + 7 )$, which suffices. 
\end{proof}

We obtain the following bounds for medium size representations.
\begin{lemma}\label{moderate_dimension_bounds}
Let $\va$ satisfy $a_n \leq \frac{2 \cdot 10^6 n}{\sigma}$
and $\sum a_j \geq
\frac{n}{\sigma\log n}$.
% Set $L=\frac{n}{\theta \log n}$ and $\eta = \lfloor n^{3/4} \rfloor$.  
Then for all sufficiently large $n$, we have the bounds

\[ \exp\left( \frac{n(5-\log \sigma)}{4\cdot 10^6 \log
n}\right) \leq d_\va \leq
\exp\left(\frac{2 \cdot 10^7 n^2}{\sigma}\right) .\]
\end{lemma}

\begin{proof}
The upper bound is immediate from Corollary \ref{crude_dimension_bound}.

For the lower bound, write
\begin{align*}
d_\va \ge d_\va^+ = \prod_{j > i} \frac{a_j + a_i + j+i-1}{j+i-1} \geq
\prod_{j > i}\left(1 + \frac{a_i + a_j}{2n}\right) . 
\end{align*}
The conditions guarantee that for $C:= \frac{2 \cdot 10^6}{\sigma}$,
$\frac{a_i +
  a_j}{2n} < C$, so that $1 + \frac{a_i +a_j}{2n} \geq \exp\left(\frac{(a_i
  + a_j)\log C}{2Cn}\right)$, as may be checked by convexity.  Thus  
\begin{align*}
d_\va \ge \prod_{j > i} \exp\left(\frac{(a_i + a_j)\log C}{2Cn}\right)
\ge \exp\left(\frac{(n-1)\log C}{Cn} \sum a_i\right)  \ge
\exp\left(\frac{n(5-\log \sigma)}{4 \cdot 10^6 \log
n}\right).
\end{align*}
\end{proof}

The preceding bounds are useful when we perform the right-side shifts
before the left ones.  Sometimes we  wish to perform left-side
shifts first.  The following lemma gives bounds in this case. 

\begin{lemma}\label{left_shift_dim_bound}
 Let $\vs \in \N^{j}$.  Let $m =
\min(j, n-j+1)$ and let $1 \leq \eta \leq m$ be a parameter. Write
\[|\vs|_{\eta, \loc} = \sum_{j - \eta \leq i \leq j} s_i.\]
 Then we have the bound
\begin{align*}
 \log \frac{d_{(\vs,\vO)+ \ve_j}}{d_{(\vs, \vO)}}& \leq m \left[ \log \frac{n
+ |\vs|_{\eta,\loc}}{m + |\vs|_{\eta,\loc}}  + \log \frac{n + j}{m+j} +2 \right]\\&
\qquad + \eta  \log (n-j+\eta) + 2(n-j+1) + \log\frac{n}{j}+ O(1) .
\end{align*}

\end{lemma}

\begin{proof}
 With  error at most $O(1)$, we have
\begin{align*}
   \log \frac{d^-_{(\vs,  \vO)+\ve_j}}{d^-_{(\vs,  \vO)}} &= \log \left[
\prod_{1 \leq i < j}\prod_{j \leq k \leq n} \left(1 + \frac{1}{k-i +  \sum_{i
< \ell \leq j} s_\ell } \right) \right]
\\ & \leq  \sum_{j-\eta \leq i < j} \sum_{j \leq k \leq n} \frac{1}{k-i
} + \sum_{i \leq j-\eta} \sum_{j \leq k \leq n}
\frac{1}{k-i + |\vs|_{\eta,\loc}}
\\& \leq \eta \log (n-j+\eta) +  m \left(1 + \log \frac{n + |\vs|_{\eta,\loc}}{m+
|\vs|_{\eta,\loc}} \right). \qquad \text{(see \eqref{sum trick} below)}
\end{align*}
Meanwhile,
\begin{align*}
  \log \frac{d^+_{(\vs,  \vO)+ \ve_j}}{d^+_{(\vs, \vO)}} &\leq \log
  \frac{d^+_{(\vO,\vO) + e_j}}{d^+_{(\vO,\vO)}}\\ 
&\leq \log \left[\prod_{1 \leq i < j} \prod_{j \leq k \leq n} \left(1
    + \frac{1}{i + 
k-1}\right) \cdot \prod_{j \leq i < k \leq n} \left(1 + \frac{1}{i + k -
1}\right) \right]
\end{align*}
The first of these terms is bounded by
\begin{equation} \frac{1}{j} + \frac{2}{j+1} + ... + \frac{m}{j + m-1}
  + m\left[\frac{1}{j+m} 
+ ... + \frac{1}{n + j-1}\right] \leq m\left[1 + \log \frac{n + j}{m+j}\right]
. \label{sum trick} \end{equation}
The second is bounded by
\begin{align*}&\frac{1}{2j} + \frac{2}{2j+1} + ... + \frac{n-j+1}{n + j} +
(n-j+1)\left[ \frac{1}{n + j} + ... + \frac{1}{2n + 2j-1}\right]
\\&\leq 2(n-j + 1). \end{align*}
Finally,
\[ \log \frac{d^0_{(\vs,  \vO)+ \ve_j}}{d^0_{(\vs, \vO)}} = \log \left[
\prod_{k \geq j} \left(1 + \frac{1}{k - \frac{1}{2}  + |\vs|}\right)\right] \leq
\log\frac{n+\frac{1}{2}}{j-\frac{1}{2}} .\]
\end{proof}

The following estimate is used in our discussion of the $L^2$ mixing time. 
\begin{lemma} \label{block_dimension}
 Let $\va(s,t) := (\vO_{n-n^t}, (n^s)_{n^t})$. Also let
 $\va(u) = \va(u,u)$. Then uniformly for $s,t$ in any compact subset of $(0,1)$,   
\begin{align*}
\log d_{\va(s,t)} = (1 - s\vee t)
  n^{s+t} (\log n+O(1)).  
\end{align*}
\end{lemma}

\begin{proof}
As usual, we write $\log d_{\va(s,t)} = \log d_{\va(s,t)}^0
+ \log d_{\va(s,t)}^+ + \log d_{\va(s,t)}^-$. The first two terms are of lower
order:
\begin{align*}
 \log d_{\va(s,t)}^0 &= \sum_{n - n^t < j \leq n} \log
\frac{n^s + j - \frac{1}{2}}{j-\frac{1}{2}} = O\left(n^{s + t -1}\right)\\
 \log d_{\va(s,t)}^+ &= \sum_{i \leq n - n^t}\sum_{n - n^t <
j \leq n}\log \frac{n^s + j + i -1}{j + i - 1} = O\left(n^{s +
t}\right), 
\end{align*}
since the logarithm is $O\left(n^{s-1}\right)$.

Meanwhile

\begin{align*}
 \log d_{\va(s,t)}^- &=
\sum_{i \leq n-n^t} \sum_{n-n^t < j \leq n} \log
\frac{n^s + j-i}{j-i}\\ 
&= \sum_{i \leq n-n^t}\sum_{j \leq n^t} \log
\frac{n^s + j+i-1}{j+i-1}\\ 
&= \sum_{k = 1}^{n^t} k \log \frac{k + n^s}{k} + n^t
\sum_{n^t < k \leq n - n^t} \log \frac{k + n^s}{k} +
\sum_{n- n^t <k \leq n} (n-k)\log \frac{k + n^s}{k}.   
\end{align*}
The last of these sums is $O(n^{s+t})$ since the logarithm
is $O(n^{s-1})$.  The first is also $O(n^{s+t})$ since it
is bounded by [use $\log (1 + x) \leq x$]
\[
 \sum_{k \leq n^t} k \log \left(1 + \frac{n^s}{k}\right)
\leq \sum_{k \leq n^t} n^s = n^{s+t}.
\]
In the middle sum, if $s > t$ then we split the sum further
at $n^s$.  The first part of the sum becomes
\[
 n^t \sum_{n^t < k \leq n^s} \log \frac{k + n^s}{k}\leq
O(n^{s+t}) + n^t \sum_{k \leq n^s} \log \frac{n^s}{k} =
O(n^{s+t}).
\]
We are left with the sum
\[
 n^t \sum_{n^{s \vee t} < k \leq n - n^t}\log \left(1 +
\frac{n^s}{k}\right) = n^t \sum_{n^{s \vee t} < k \leq
n-n^t} \left(\frac{n^s}{k} + O\left(\frac{n^{2s}}{k^2}\right)\right)
\]
which evaluates to the main term, as desired.
\end{proof}

We use the following estimate in the $L^2$ upper bound.

\begin{lemma}\label{dimension_clump_bound}
Let $C$ be a constant, let $\delta = \frac{1}{\log n}$, let $1 \leq S \leq \frac{n^{\frac{1}{2}}}{1 + \log n}$ and let $B$ be the collection of indices
\[
 B = \left\{j \in \left[n- \frac{Cn}{\log n}, n\right]: a_j \in \left(\left(1 - \delta\right)S,
S\right]\right\}.
\] Suppose that
\[\frac{n^{\frac{1}{2}}}{\log n}\leq |B| \leq \frac{Cn}{\log n}.\]  We have
\[
 \log d_\va(B) \stackrel{def}{=} \sum_{j \in B} \log
d_\va(j) \leq \left(1 - \frac{\log |B|}{\log n}\right)
|B|S(\log n + O(1)).
\]
\end{lemma}

\begin{proof}

Let $m$ be the
last index in $B$, and, with an eye toward applying the
block dimension bound in Lemma \ref{block_dimension}, set $T
= |B|$.  Obviously those indices greater than $m$ do not
affect $d_\va(B)$, so henceforth we  write $\va$ in
place of $\va(m)$, the weight  truncated at $m$
corresponding to a representation on $SO(2m+1)$.

Factor $d_\va(B)$ as $d_\va(B) = E_\va(B) \cdot I_\va(B)$
where $E_\va(B)$ corresponds to factors pairing indices in
$B$ with those outside,
\[
 E_\va(B) = \prod_{j \in B} \frac{\ta_j}{j-\frac{1}{2}} \prod_{i
\leq m-T} \frac{\ta_j^2 - \ta_i^2}{(j-\frac{1}{2})^2 - (i-\frac{1}{2})^2},
\]
and where $I_\va(B)$ pairs indices both inside $B$,
\[
 I_\va(B) = \prod_{j \in B} \prod_{m-T \leq i < j}
\frac{\ta_j^2 - \ta_i^2}{(j-\frac{1}{2})^2 - (i-\frac{1}{2})^2}.
\]
Set $s = \frac{\log S}{\log m}$, $t = \frac{\log
T}{\log m}$ and write as before $\va(S, T) =
(\vO_{m-m^t}, (m^s)_{m^t})$.  Then we may bound $E_\va(B)
\leq d_{\va(s, t)}$ so that Lemma
\ref{block_dimension} yields the bound
\[
 \log E_\va(B) \leq (1 - s\vee t)ST (\log n + O(1)).
\]
This is of the right size, so we now work to show that
$\log I_\va(B)$ is of lower order by a factor of $\log n$.

Write $I_\va(B) = I_\va(B)^+ I_\va(B)^-,$ where
\[
   I_\va(B)^+ =
\prod_{j \in B} \prod_{m-T< i < j} \frac{ a_i + a_j +i +
j -1}{i + j -1}, \quad I_\va(B)^- = \prod_{j \in B}
\prod_{m-T< i < j} \frac{a_j - a_i + j-
i }{j-i}
\]
Then we have the bound
\begin{align*}
 \log I_\va(B)^+ \leq \sum_{j \in B} \sum_{m-T <i < j}
\log \left(1 + \frac{S}{n\left(1-O\left(\frac{1}{\log n}\right)\right)}\right) \leq \left(1 +
O\left(\frac{1}{\log n}\right)\right) \frac{ST^2}{2n},
\end{align*}
which more than suffices, since $T \leq \frac{Cn}{\log n}$. To bound
$I_\va(B)^-$ consider the representation
$\rho_\vb$ on $SO(2T+1)$ with $b_i = a_{m-T + i} -
a_{m-T}\leq \delta S$.  Then $I_\va(B)^- = d_{\vb}^- \leq
d_\vb$. Since $S \leq T$, the bound
(\ref{dimension_increment_bound}) gives [use $\log(1 + x) \leq x$]
\[
 \log d_\vb \leq \sum_{i=1}^T b_i (\log T + O(1)) \leq \delta ST
(\log n + O(1)).
\]
Since $\delta = \frac{1}{\log n}$ this completes the proof.
\end{proof}

It is also convenient to have some lower bounds for dimensions.  The following bound is a useful reference point.

\begin{lemma}\label{first_dimension_lower_bound}
 There is $c > 0$ such that
\[d_\va \geq \exp\left(c \min\left(n, \sum_j a_j\right)\right).\] 
\end{lemma}
\begin{proof}
 Write $\va \leftrightarrow \vs$ and write
 \[
  d_{\vs}^+ \geq \prod_{k=1}^n \prod_{1 \leq j \leq \frac{k}{2}} \left(1 + \frac{\sum_{i \leq k} s_i}{j+k-1}\right) \geq \exp\left( c \sum_{k = 1}^n \min(a_k, k)\right).
 \]
This suffices, since if $a_k \geq k$ for any $k$, then $a_\ell \geq k$ for all $\ell > k$, whence a lower bound of $\exp(c' (n-k)k)$.
\end{proof}

The next lemma is useful for proving Proposition \ref{sum_of_dimension_bound}.

\begin{lemma}\label{dimension_lower_bound}
 Let $\eta$, $1 \leq \eta \leq n$ be a parameter, and  set $\N^n = \N^{n-\eta}
\oplus
\N^{\eta}$.  We have
\[d_{(\vs_1, \vs_2)} \geq d_{(\vs_1, \vO)}^+ d_{(\vO, \vs_2)}^-.\]  Set
$|\vs_1| = \sum_{i \leq n - \eta} s_i$.  We have
\[ d^+_{(\vs_1, \vO)} \geq \left(1 + \frac{|\vs_1|}{2n}\right)^{n \eta
  - \eta^2}.\]  In particular, 
 for $|\vs_1|< n$, and for $\eta < n^{1-\epsilon}$ and all
 sufficiently large $n$,  
\[ d_{(\vs_1, \vO)}^+ \geq e^{\frac{|s_1| \eta}{3}}. \]
Define $|\vs_2|_p = \sum_{i < \eta} ((\eta - i + 1)s_i)$.  Then
\[ d_{(\vO, \vs_2)}^- \geq \prod_{1 \leq j < n - \eta} \prod_{1 \leq k \leq
\eta}\left(1 +
\frac{\sum_{i \leq k} s_i}{\eta + j}\right)\geq \prod_{1 \leq j <
n-\eta}\left(1 
+
\frac{|\vs_2|_p}{\eta + j}\right).\]
\end{lemma}
\begin{proof}
 The identity $d_{(\vs_1, \vs_2)} \geq d_{(\vs_1, \vO)}^+ d_{(\vO, \vs_2)}^-$
is immediate from the definitions.

Now we have
\begin{align*}
 d_{(\vs_1, \vO)}^+ &= \prod_{1 \leq j < k \leq n}\left( \frac{\left[
       \sum_{i \leq 
j} s_i + \sum_{i \leq k}s_i\right] + j+ k - 1}{j+k-1}\right)
\\& \geq \prod_{j < n-\eta} \prod_{n-\eta \leq k \leq n} \left(\frac{j+k-1 +
|\vs_1|}{j+k-1}\right)
 \geq \left(1 +
\frac{|\vs_1|}{2n}\right)^{n \eta - \eta^2},
\end{align*}
proving the bound for $d_{(\vs_1, \vO)}^+$ for sufficiently large $n$.
Meanwhile, denoting $s_i$ the components of $\vs_2$,
\begin{align*}
 d_{(\vO, \vs_2)}^- &\geq \prod_{1 \leq j < n-\eta} \prod_{1 \leq k \leq
\eta}\left(1 + \frac{\sum_{i \leq k} s_i}{k + n-\eta - j}\right)
\\ & \geq \prod_{1 \leq j < n - \eta} \prod_{1 \leq k \leq \eta}\left(1 +
\frac{\sum_{i \leq k} s_i}{\eta + j}\right) \qquad \text{(switching $j
\leftrightarrow n- \eta -j$)}\\ 
& \geq \prod_{1 \leq j < n-\eta}\left(1 + \frac{\sum_{1 \leq k \leq \eta}
\sum_{i \leq k} s_i}{\eta + j}\right) = \prod_{1 \leq j < n-\eta}\left(1 +
\frac{|\vs_2|_p}{\eta + j}\right).
\end{align*}
\end{proof}

We conclude the section by proving Proposition \ref{sum_of_dimension_bound}.
\begin{proof}[Proof of Proposition \ref{sum_of_dimension_bound}]
Set $\eta = \left\lfloor n^{3/4} \right\rfloor$. Then letting $\N^n = \N^{n-\eta} \oplus
\N^{\eta}$ observe
$d_{(\vs_1, \vs_2)} \geq d_{(\vs_1, \vO)}^+ d_{(\vO, \vs_2)}^-$, $d_\vO =1$,
and therefore
\begin{align*} \sum_{\vO \neq \vs \in \N^{n} } d_\vs^{\frac{-c}{\log n}} &=
-1+\sum_{\vs} d_\vs^{\frac{-c}{\log n}}
 \\&
 \leq -1 + \left(\sum_{\vs_1 \in \N^{n-\eta}} \left(d_{(\vs_1,
\vO)}^+\right)^{\frac{-c}{\log n}}\right)\cdot \left( \sum_{\vs_2 \in
\N^{\eta}}\left(d_{(\vO, \vs_2)}^-\right)^{\frac{-c}{ \log n}}\right),
\end{align*}
so it suffices to show that each sum on the right is $1 +
O(e^{-c})$ as $n \to 
\infty$.

We first handle the sum over $\vs_1$. Observe that $\# \{\vs \in \N^n:
|\vs|=j\} \le n^j \wedge j^n$. Applying the bound of Lemma 
\ref{dimension_lower_bound},
\begin{align*}
 \sum_{\vs_1 \in \N^\eta} (d_{(\vs_1, \vO)}^+)^{\frac{-c}{\log n}} &\leq 1 + \sum_{j =
1}^\infty \#\{\vs \in \N^{n-\eta}: |\vs| = j\}\left(1 +
\frac{j}{2n}\right)^{ \frac{-cn^{\frac{7}{4}}}{\log n}} 
\\ & \leq 1 + \sum_{j
=1}^{4n} n^j e^{\frac{-c'j n^{\frac{3}{4}}}{\log n}} + \sum_{j = 4n}^\infty j^n
\left(\frac{2j}{2n}\right)^{\frac{-c n^{\frac{7}{4}}}{\log n}} 
\\& = 1 + o(1), \qquad n \to \infty.
\end{align*}

Now we consider the sum over $\vs_2$.  Again applying the bound of the
lemma,
\[ \sum_{\vs_2 \in \N^\eta} \left(d_{(\vO, \vs_2)}^-\right)^{\frac{-c}{\log
  n}} \leq 1 + 
\sum_{q = 1}^\infty \#\{\vs \in \N^\eta: |\vs|_p = q\} \left[ \prod_{j
    < n-\eta} 
\left(1 + \frac{q}{\eta + j}\right)\right]^{\frac{-c}{\log n}}.\]
Evidently $\#\{\vs \in \N^\eta: |\vs|_p = q\}$ is the number of partitions of
$q$ into parts of size at most $\eta$.  This is bounded by the total number of
partitions of $q$, which is $e^{O(\sqrt{q})}$, and by $q^\eta$.  Thus the above
sum over $q$ is bounded by
\begin{align*}&\sum_{1 \leq q < \eta/2} \exp\left(O(\sqrt{q}) -
    \frac{cq (\log n - \log \eta)}{2\log n}\right) 
\\&\qquad + \sum_{q = \eta/2}^{n^{5/4}} \exp\left( O(\sqrt{q}) - \frac{c\eta}{8}\right)
\\& \qquad + \sum_{q = n^{5/4}}^\infty \exp\left(\eta \log q -
  \frac{c'(n-\eta)\log q}{ \log n}\right) , \qquad c' > 0 
\end{align*} the first term is $O\left(e^{\frac{-c}{8}}\right)$ and the remaining terms
are $o(1)$ as $n \to \infty$. 

\end{proof}

\section{Contour formula for characters of
$SO(2n+1)$} \label{general contour SO(N)}
  We now generalize our integral formula for the character ratio at a rotation
 to give a $k$-fold integral formula for the
character ratios evaluated at a generic element of
$SO(2k)\subset SO(2n+1)$.  

\begin{theorem}
Denote $\{L_i, 1 \leq i \leq n\}$ the roots of $SO(2n+1)$.  Let
$\rho_\va$, $\va = a_1 \geq a_2 \geq\cdots \geq a_n \geq 0$, be
an irreducible representation of $SO(2n+1)$ corresponding to
highest weight $\sum a_i L_i$, and for $1 \leq i \leq n$, write $\tilde{a}_i = a_i +i - \frac{1}{2}$.   Let
$\vtheta = (\theta_1, ..., \theta_k) \in (\bR/2\pi \zed)^k$, $\theta_i \neq 0$, $\theta_i \neq \theta_j$ for $i < j$
be an element of a torus contained in $SO(2k) \subset
SO(2n+1)$. 
We have the integral character ratio formula
\[
 r_\va(\vtheta) = \frac{r_\va(\vtheta)}{d_\va} =
L(\vtheta) \operatorname*{\oint \cdots
\oint}_{\mathscr{C}_1 \times \cdots \times \mathscr{C}_k}
M_{\va; \vtheta}(\vz) d\vz,
\]
where the contours $\mathscr{C}_i$ are
chosen to have winding number 1 about each of the poles $\pm
\ta_i$ of the integrand
\[
 M_{\va; \vtheta}(\vz) = \prod_{j
= 1}^k M_{\va; \theta_j}(z_j)\times  \prod_{1 \leq j < \ell
\leq
k} \left(z_\ell^2 - z_j^2\right); \qquad M_{\va, \theta}(z) = \frac{\sin \theta z}{\prod_{j=1}^n \left(\tilde{a}_j^2 - z^2\right)}.
\] The leading term $L(\vtheta)$ is 
independent of the representation $\va$ and given by
\[
L(\vtheta) =
\frac{(2n-1)!(2n-3)!\cdots(2n-2k+1)!}{2^{2nk
-k^2}\prod_{j=1}^k 
\sin\left(\frac{\theta_j}{2}\right)^{2n-2k+1} \prod_{1 \leq i < j \leq k}
\left(\sin\left(\frac{\theta_i}{2}\right)^2 -
\sin\left(\frac{\theta_j}{2}\right)^2\right)}.
\]
 \end{theorem}

 \begin{proof}

 For $\vtheta \in \bT^n$ in general position on the maximal torus, 
the character
$\chi_\va(\vtheta)$ is given by the following determinantal
formula (\cite{fulton_harris}, p. 408, (24.28))
\footnote{$\tilde{0}_i = i-\frac{1}{2}$. Henceforth we write only $\rs$ for $\sin$. }
\[
 \chi_\va(\vtheta) = \left| 
\begin{array}{cccc}
\rs\left(\ta_1 \theta_1\right) & \rs\left(\ta_1 \theta_2\right) & \cdot & \rs
\left(\ta_1 \theta_n\right) \\
\rs\left(\ta_2 \theta_1\right) & \rs\left(\ta_2 \theta_2\right) & \cdot &
\rs\left(\ta_2 \theta_n\right) \\
 \cdot &\cdot &&\cdot  \\
\rs\left(\ta_n \theta_1\right) & \rs\left(\ta_n \theta_2\right) & \cdot & \rs
\left(\ta_n \theta_n\right)
\end{array}
\right|\Bigg/ 
\left| 
\begin{array}{cccc}
\rs\left(\tO_1 \theta_1\right) & \rs\left(\tO_1 \theta_2\right) & \cdot & \rs
\left(\tO_1 \theta_n\right) \\
\rs\left(\tO_2 \theta_1\right) & \rs\left(\tO_2 \theta_2\right) & \cdot &
\rs\left(\tO_2 \theta_n\right) \\
\cdot &\cdot &&\cdot  \\
\rs\left(\tO_n \theta_1\right) & \rs\left(\tO_n \theta_2\right) & \cdot & \rs
\left(\tO_n \theta_n\right)
\end{array}
\right|
\]
Since the dimension is equal to the character value at
$\vtheta = \vO$, we obtain the character
ratio at $\vtheta \in \bT^k$,
\begin{align*}
 r_\va(\vtheta) &= 
 \frac
{
\left| 
\begin{array}{cccccc}
\rs\left(\ta_1 \theta_1\right) &  \cdot& \rs
\left(\ta_1 \theta_k\right) & \rs\left(\ta_1 \epsilon_{k+1}\right) & \cdot&
\rs\left(\ta_1 \epsilon_{n}\right) \\
\rs\left(\ta_2 \theta_1\right) &  \cdot& \rs
\left(\ta_2 \theta_k\right) & \rs\left(\ta_2 \epsilon_{k+1}\right) & \cdot&
\rs\left(\ta_2 \epsilon_{n}\right)  \\
\cdot & &\cdot &\cdot & & \cdot\\
\rs\left(\ta_n \theta_1\right) &  \cdot& \rs
\left(\ta_n \theta_k\right) & \rs\left(\ta_n \epsilon_{k+1}\right) & \cdot&
\rs\left(\ta_n \epsilon_{n}\right) 
\end{array}
\right|
}
{
\left| 
\begin{array}{cccc}
\rs\left(\ta_1 \epsilon_{1}\right) & \rs\left(\ta_1 \epsilon_{2}\right) & \cdot&
\rs
\left(\ta_1 \epsilon_{n}\right) \\
\rs\left(\ta_2 \epsilon_{1}\right) & \rs\left(\ta_2 \epsilon_{2}\right) & \cdot&
\rs\left(\ta_2 \epsilon_{n}\right) \\
\cdot && & \cdot\\
\rs\left(\ta_n \epsilon_{1}\right) & \rs\left(\ta_n \epsilon_{2}\right) & \cdot&
\rs
\left(\ta_n \epsilon_{n}\right)
\end{array}
\right| 
} \times
\\
&\qquad \qquad\qquad \qquad \times 
\frac
{
\left| 
\begin{array}{cccc}
\rs\left(\tO_1 \epsilon_{1}\right) & \rs\left(\tO_1 \epsilon_{2}\right) & \cdot&
\rs\left(\tO_1 \epsilon_{n}\right) \\
\rs\left(\tO_2 \epsilon_{1}\right) & \rs\left(\tO_2 \epsilon_{2}\right) & \cdot&
\rs\left(\tO_2 \epsilon_{n}\right) \\
\cdot &&& \cdot\\
\rs\left(\tO_n \epsilon_{1}\right) & \rs\left(\tO_n \epsilon_{2}\right) & \cdot&
\rs\left(\tO_n \epsilon_{n}\right)
\end{array}
\right|
}{
\left| 
\begin{array}{cccccc}
\rs\left(\tO_1 \theta_1\right) &  \cdot& \rs
\left(\tO_1 \theta_k\right) & \rs\left(\tO_1 \epsilon_{k+1}\right) & \cdot&
\rs\left(\tO_1 \epsilon_{n}\right) \\
\rs\left(\tO_2 \theta_1\right) &  \cdot& \rs
\left(\tO_2 \theta_k\right) & \rs\left(\tO_2 \epsilon_{k+1}\right) & \cdot&
\rs\left(\tO_2 \epsilon_{n}\right)  \\
\cdot &&\cdot &\cdot && \cdot\\
\rs\left(\tO_n \theta_1\right) &  \cdot& \rs
\left(\tO_n \theta_k\right) & \rs\left(\tO_n \epsilon_{k+1}\right) & \cdot&
\rs\left(\tO_n \epsilon_{n}\right) 
\end{array}
\right|
},
\end{align*}
where $\epsilon_1, ..., \epsilon_n$ represent
infinitesimals tending to 0.

Let $\epsilon_1 \succ \epsilon_2 \succ\cdots \succ \epsilon_n$
where $a \succ b$ means $b$ is eventually smaller than any
fixed power of $a$ as $a$ and $b$ tend to 0. The reason for this choice is that, when determinants involving $\bepsilon$ are expanded as alternating sums below, we are able to neglect all but a single term.  

In the terms containing epsilon arguments, we expand 
$\sin$ in its power series about 0 performing row reduction to eliminate lower order terms, then take the limit to eliminate higher powers, with the result that we may write $r_\va =\frac{ R_\va}{R_\vO}$ where
\[
 R_\va = 
\frac{\left|
\begin{array}{ccccccc}
 \frac{\rs\left(\ta_1 \theta_1\right)}{\ta_1} & \cdot &
\frac{\rs\left(\ta_1\theta_k\right)}{\ta_1} & \left(-\ta_{1}^2\right)^{n-k-1}
& \left(-\ta_{1}^2\right)^{n-k-2} & \cdot & 1 \\
  \frac{\rs\left(\ta_2 \theta_1\right)}{\ta_2} & \cdot &
\frac{\rs\left(\ta_2\theta_k\right)}{\ta_2} & \left(-\ta_{2}^{2}\right)^{n-k-1}
& \left(-\ta_{2}^{2}\right)^{n-k-2} & \cdot & 1 \\
\cdot &&\cdot & \cdot &\cdot && \cdot
\\
 \frac{\rs\left(\ta_n \theta_1\right)}{\ta_n} & \cdot &
\frac{\rs\left(\ta_n\theta_k\right)}{\ta_n} & \left(-\ta_{n}^{2}\right)^{n-k-1}
& \left(-\ta_n^{2}\right)^{n-k-2} & \cdot & 1 
\end{array}
\right|
}
{\left|
\begin{array}{cccc}
 \left(-\ta_1^{2}\right)^{n-1} & \left(-\ta_1^{2}\right)^{n-2} & \cdot & 1\\
 \left(-\ta_2^{2}\right)^{n-1} & \left(-\ta_2^{2}\right)^{n-2} & \cdot &1\\
 \cdot & && \cdot \\
 \left(-\ta_n^{2}\right)^{n-1} & \left(-\ta_n^{2}\right)^{n-2} & \cdot & 1
\end{array}
\right|
}.
\]

We observe that the denominator is the Vandermonde
$ \prod_{1 \leq i <j \leq n} \left(\ta_i^2 - \ta_j^2\right).$  Expanding the first $k$ columns of the numerator, and dividing by the Vandermonde in the denominator, we obtain (note that the $\ta_i$ are increasing)
\begin{align*}
  R_\va(\vtheta) &=
\sum_{\substack{1 \leq i_1, ..., i_k \leq
n\\ \text{distinct}}}\frac{ (-1)^{ \sum_{1 \leq j < \ell \leq k} \one(i_j < i_\ell)} \prod_{j=1}^k \sin\left(\theta_j
\ta_{i_j}\right) \prod_{1 \leq j < \ell \leq k} \left|\ta_{i_j}^2 - \ta_{i_\ell}^2 \right| }{\prod_{j=1}^k\left( \ta_{i_j}\prod_{\ell
\neq i_j} \left(\ta_\ell^2 -\ta_{i_j}^2  \right)\right)}\\
&=
\sum_{\substack{1 \leq i_1, ..., i_k \leq
n\\ \text{distinct}}} \frac{\prod_{j=1}^k \sin\left(\theta_j
\ta_{i_j}\right)\prod_{1 \leq j < \ell \leq k} \left(\ta_{i_j}^2 -
\ta_{i_\ell}^2\right)} {\prod_{j=1}^k\left( \ta_{i_j}\prod_{\ell
\neq i_j} \left(\ta_\ell^2 -\ta_{i_j}^2  \right)\right)} = (-1)^{\binom{k}{2}}\oint
\cdots \oint M_{\va, \vtheta}(\vz) d\vz,
\end{align*}
since the factor $\prod_{1 \leq j < \ell  \leq k} \left(z_\ell^2 -
z_j^2\right)$ in $M_{\va, \vtheta}(\vz)$ produces the
 factor  \[(-1)^{\binom{k}{2}} \prod_{1 \leq j < \ell
\leq k} \left(\ta_{i_j}^2 - \ta_{i_\ell}^2\right)\] at each set of
poles.

Backtracking a bit, we determine the constant term by
\begin{align*}
 R_\vO &=  \frac{\epsilon_1^{2n-1}
\epsilon_2^{2n-3}\cdots\epsilon_k^{2n-2k+1}}{
(2n-1)!(2n-3)!\cdots(2n-2k+1)!}\frac{\det \left[\rs\left(\tO_i
\theta_j\right)\Big|_{j\leq k}\rs\left(\tO_i
\epsilon_{j}\right)\right]}
{\det\left[\rs\left(\tO_i
\epsilon_{j}\right)\right]}
\end{align*}
Since $\rs\left(\left(j-\frac{1}{2}\right)\theta\right)$ is a polynomial in
$\rs\left(\frac{\theta}{2}\right)$ with highest order term $(-4)^{j-1} \rs\left(
\frac{\theta}{2}\right)^{2j-1}$, performing row reductions we obtain
\[
 R_\vO = \frac{\epsilon_1^{2n-1}
\epsilon_2^{2n-3}\cdots\epsilon_k^{2n-2k+1}}{
(2n-1)!(2n-3)!\cdots(2n-2k+1)!}\frac{\det \left[\rs\left(\frac{
\theta_j}{2}\right)^{2i-1}\Big|_{j\leq k}\rs\left(
\frac{\epsilon_j}{2}\right)^{2i-1}\right]}
{\det\left[\rs\left(\frac{
\epsilon_j}{2}\right)^{2i-1}\right]}.
\]
Keeping in mind the relative size of the infinitesimals and expanding the determinants alternating sums we see that, up to sign, the numerator contributes the determinant of its lower $k\times k$ minor times a product of infinitesimals, while the denominator contributes a single product of infinitesimals.  This reduces the formula  to
\begin{align*}
 R_\vO &= \frac{(-1)^{\binom{k}{2}}\epsilon_1^{2n-1}\cdots\epsilon_k^{2n-2k+1}}{s\left(\frac{\epsilon_1}{2}\right)^{2n-1}
\cdots s\left(\frac{\epsilon_k}{2}\right)^{2n-2k+1}}\frac{\prod_{i=1}^k \left(s\left( \frac{\theta_i}{2} \right)^{2n - 2k + 1}\right)\det \left[s\left(\frac{\theta_j}{2} \right)^{2(i-1)} \right]_{i,j = 1}^k}{
(2n-1)!(2n-3)!\cdots(2n-2k+1)!}
\\&= \frac{(-1)^{\binom{k}{2}} 2^{2nk-k^2} \prod_{i=1}^k \left(s\left(\frac{\theta_i}{2} \right)^{2n - 2k + 1} \right)\prod_{1 \leq i < j \leq k} \left(s\left(\frac{\theta_i}{2}\right)^2 - s\left(\frac{\theta_j}{2}\right)^2\right)}{(2n-1)!(2n-3)!\cdots(2n-2k+1)!}.
\end{align*}

\end{proof}

\bibliography{Rosenthal}{}
\bibliographystyle{plain}

\end{document}